\documentclass[11pt]{article}

\usepackage{amssymb}
\usepackage{graphicx}
\usepackage{color}
\usepackage{amsmath,amsthm,amsfonts,epsfig,setspace}
\numberwithin{equation}{section}

\newcommand{\ds}{\displaystyle}
\def\nm{\noalign{\medskip}}
\newtheorem{thm}{Theorem}[section]
\newtheorem{rmk}{Remark}[section]

\newtheorem{definition}{Definition}[section]
\newtheorem{lem}{Lemma}[section]

\newtheorem{prop}{Propsition}[section]
\newtheorem{asump}{Assumption}[section]

 \def\p{\partial}
\def \Vh0{\stackrel{\circ}{V}_h} 
   
\def\Om{\Omega}  \def\om{\omega}
   
\def\l{\label}  \def\f{\frac}  

\def\e{\epsilon}

\def\l|{\left|}
\def\r|{\right|}

\newcommand{\R}{\mathbb{R}}

\newcommand{\lc}
{\mathrel{\raise2pt\hbox{${\mathop<\limits_{\raise1pt\hbox
{\mbox{$\sim$}}}}$}}}

\newcommand{\gc}
{\mathrel{\raise2pt\hbox{${\mathop>\limits_{\raise1pt\hbox{\mbox{$\sim$}}}}$}}}

\newcommand{\ec}
{\mathrel{\raise2pt\hbox{${\mathop=\limits_{\raise1pt\hbox{\mbox{$\sim$}}}}$}}}

\def\be{\begin{equation}} \def\ee{\end{equation}}

\def\bea{\begin{eqnarray}}  \def\eea{\end{eqnarray}}

\def\beas{\begin{eqnarray*}} \def\eeas{\end{eqnarray*}}

\def\bn{\begin{enumerate}} \def\en{\end{enumerate}}

\def\bd{\begin{description}} \def\ed{\end{description}}

\title{ A mathematical theory of super-resolution by using a system of sub-wavelength Helmholtz resonators\thanks{\footnotesize This work was supported  by the
ERC Advanced Grant Project MULTIMOD--267184.}}

\author{
Habib Ammari\thanks{\footnotesize Department of Mathematics and Applications,
Ecole Normale Sup\'erieure, 45 Rue d'Ulm, 75005 Paris, France
(habib.ammari@ens.fr, hai.zhang@ens.fr).} \and   Hai Zhang\footnotemark[2]
}

\begin{document}
\maketitle

\begin{abstract}
A rigorous mathematical theory is developed to explain the super-resolution phenomenon observed in the experiment \cite{LFL2011}. A key ingredient is the calculation of the resonances and the Green function in the half space with the presence of a system of Helmholtz resonators in the quasi-stationary regime. By using boundary integral equations and generalized Rouche's theorem, the existence and the leading asymptotic of the resonances are rigorously derived. The integral equation formulation also yields the leading order terms in the asymptotics of the Green function. The methodology developed in the paper provides an elegant and systematic way for calculating resonant frequencies for Helmholtz resonators in assorted space settings as well as in various frequency regimes. By using the asymptotics of the Green function, the analysis of the imaging functional of the time-reversal wave fields becomes possible, which clearly demonstrates the super-resolution property. The result provides the first mathematical theory of super-resolution in the context of a deterministic medium and sheds light to the  mechanism of super-resolution and super-focusing for waves in deterministic complex media.
\end{abstract}

\medskip

\bigskip

\noindent {\footnotesize Mathematics Subject Classification
(MSC2000): 35R30, 35B30.}

\noindent {\footnotesize Keywords: super-resolution, diffraction limit,
sub-wavelength-scaled resonant medium,  Helmholtz resonator, time-reversal.}

\tableofcontents

\section{Introduction} \label{sec-intro}

When light is focused by the objective of a microscope, the notion of light rays converging to an infinitely sharp focal point does not hold true. Instead, as observed by Abbe in 1873 \cite{abbe}, the light wave forms a blurred or diffracted focal spot with a finite size due to diffraction.
The size of the spot depends on the wavelength of the light and the angle at which the light wave converges; the latter is, in turn, determined by the numerical aperture of the objective.
Similarly, a point emitter also appears as a blurred spot, and the size of the spot places a fundamental limit on the minimal distance at which we can resolve two emitters.
 The intensity profile of this spot, which defines the point spread function of the microscope, has approximately the same width as that of the focal spot described above. Consequently, two identical emitters separated by a distance less than the width of the point spread function will appear as a single object, making them unresolvable from each other \cite{book_imaging, cell}. This resolution limit, referred to as the Abbe-Rayleigh or the diffraction limit of resolution,  applies only to light that has propagated for a distance substantially larger than its wavelength \cite{gang1,gang2}.  It is well-known since the seminal work of Synge \cite{synge} that near-field microscopes achieve resolutions well below the diffraction limit.

Discerning features that are spectrally disparate is not challenging by diffraction.
It is now well-established that spectroscopic imaging can yield super-resolution \cite{nature2}. Likewise,
Abbe's barrier  does not prevent finding out the location of a point emitter with arbitrary  precision \cite{candes1,candes2, science}. Breaking Abbe's barrier is only about discerning features within a distance smaller than Abbe's barrier.

Since the mid-20th century, several approaches aimed at pushing the diffraction limits by reducing the focal spot size.  In the optical domain, the sub-wavelength-scaled resonant
media capable of beating the classical diffraction limit and the concepts such as superlenses
\cite{apl}, imaging single molecules \cite{science}, and super-oscillations \cite{berry}, could provide feasible alternatives \cite{nature4}.

Artificially engineered metamaterials offer the possibility of building a superlens that overcomes the diffraction limit. The limitation of present designs of the far-field superlens is that the object must be in the near field of the superlens, although the image can be projected into the far-field. This is because one has to make sure that the evanescent waves do not decay too much before reaching the superlens and being enhanced or converted into propagating waves. Another fundamental challenge is the loss, as most superlens schemes involve resonances in metal, which limits both the practical resolution and transmission \cite{apl, nature1, nature3}.

Imaging single molecules rely on the principle that a single emitter can be localized with a high accuracy if the signal-to-noise ratio in the data is sufficiently high. Breaking Abbe's limit is  possible if the features can be recorded sequentially \cite{science}. Single-molecule-based super-resolution techniques use photoswitching or other mechanisms to stochastically activate individual molecules within the diffraction-limited region at different times.  Images with subdiffraction limit resolution are then reconstructed from the measured positions of individual fluorophores \cite{nature2}.

Superoscillation is the fact that band-limited functions are able to oscillate arbitrarily faster than the highest Fourier components they contain \cite{berry}.
The persistence of superoscillations can be interpreted as the propagation of sub-wavelength structure farther into the field than the more familiar evanescent waves. By using this concept,  examples of sub-wavelength localizations of light generated by a nano-hole array and a thin meta-dielectric shell have been demonstrated recently \cite{nanolett}.

In this paper,  we mathematically investigate the mechanism underlying the resolution enhancement using sub-wavelength-scaled resonant media \cite{LFL2011, lemoult11, lemoult10}.  The main focus is on the possibility to break the diffraction barrier from the far-field measurements using time-reversal.
The principal of time-reversal is to take advantage of the reversibility of lossless waves and using wave equation in order to back-propagate signals to the sources that emitted them; see \cite{fin-99, F08}. The idea is to measure the emitted wave in the far-field and retransmit it through the background medium in a time-reversed chronology. The Helmholtz-Kirchhoff identity shows that the resolution is determined by the behavior of the imaginary part of the Green function \cite{book_imaging, lnm2}.

Many interesting mathematical works have dealt with different aspects of time-reversal phenomena: see, for instance,  \cite{bardos, BF02} for time-reversal in the time-domain, \cite{chambers, nachman} for frequency domain counterpart of time-reversal, \cite{BPZ00, layered_book} for time-reversal in random media, and \cite{abdul} for time-reversal in attenuating media.  It is proved in \cite{yves}  that using structured (i.e., periodic) media
can improve the resolution in imaging from far-field measurements. Indeed,
 the resolution enhancement in terms of the material parameters
and the geometry of the structured medium can be precisely quantified.

The aim here is to develop a
rigorous mathematical theory to explain the super-resolution phenomenon observed in
the experiment \cite{LFL2011}. The mechanism for resolution enhancement is completely different from the one in structured media. Moreover, the resolution enhancement is dramatically larger
\cite{lerosey07}.

A key ingredient in the proposed theory is a novel calculation of resonances and the Green function in half space with the presence of a system of Helmholtz resonators in the quasi-stationary regime. The theory of Helmholtz resonators has a long history. Rayleigh \cite{rayleigh} first showed that for some frequencies close to zero, the field scattered by the resonator is significantly different from the field scattered by the closed resonator.
Then Miles \cite{miles} showed numerically that  the same phenomena occurs also in a neighborhood of any Neumann eigenvalue of the Laplacian in the closed resonator.
In \cite{G1997}, using the matched asymptotic method, asymptotic expansions of resonances
associated with Helmholtz resonators are obtained and their rigorous justification
is established by sophisticated functional analysis arguments.
 Here, we prove existence and derive high-order leading terms in the asymptotic expansions of the resonances using layer potential techniques and generalized Rouche's theorem in the same spirit as in \cite{AKL2009}. The integral equation formulation also yields the leading order terms in the asymptotic expansions of the Green function. Based on this, the analysis of the imaging functional of  time-reversal wave fields becomes possible, which clearly demonstrates the super-resolution property and provides the first mathematical justification of super-resolution in the context of a deterministic medium and sheds light on the  mechanism of super-resolution (or super-focusing) for wave fields in deterministic complex media. The analysis in this paper also opens many new avenues for mathematical imaging and focusing in resonant media. Resonant media may be used to shape, compress, and control waves at a sub-wavelength scale \cite{LFL2011, nature3}.

The paper is organized as follows. Section \ref{time-reversal} is devoted to state significant results in this paper. We first introduce time-reversal in homogeneous media. Then, we present a simplified model for the time-reversal experiment with Helmholtz resonators in \cite{LFL2011}. Finally, some notation and key results are collected. In Section \ref{sec-potential}, we introduce some basic properties of the Riesz potential defined on a flat surface in $\R^3$, which turns out to be crucial to the subsequent analysis. In Section \ref{sec-resonator-single}, we present a novel method to calculate resonances and the Green function for a single Helmholtz resonator. We formulate the scattering problem as a boundary integral equation defined on the opening of the resonator, which depends on the frequency. The resonances of the resonator becomes equivalent to the characteristic values of the operator-valued analytic function.
We decompose the boundary integral operator into several parts. The key part turns out to be the one related to Riesz potential defined on the opening. Useful scaling properties are introduced to study these operators, which allow us to apply generalized Rouche's theorem to calculate the leading asymptotic of the resonances in the quasi-stationary regime and also the Green function.
In Section \ref{sec-resonator-multi}, we study the resonances and Green's function for a system of Helmholtz resonators, which prove Theorem \ref{thm-multi}. The presentation resembles closely the one for the case of a single resonator and can be regarded as a generalization and application of the approach developed in Section \ref{sec-resonator-single}. Finally, in Section \ref{sec-SR-proof}, we prove Theorems  \ref{thm-super-resolution-1} and \ref{thm-super-resolution-2} on super-resolution. The paper ends with a short discussion.

\section{Time-reversal and main results} \label{time-reversal}

\subsection{Time-reversal in the homogenous space} \label{sec-sub-TR}
We present a general setup for the time-reversal experiments in the homogeneous space.

Let $f(t)\in L^2(0, \infty)$ be a signal which has compact support in $\R$. Let
$\Om$ be a domain in $\R^3$, bounded or unbounded.
Let $x_0 \in \Om$ be a given point. Consider the time domain scattering problem
\begin{eqnarray}
u_{tt}(x, t)-\Delta u(x,t)&=& \delta(x-x_0)f(t),   \quad (x,t)\in \Om \times(0, \infty),  \label{wave1}\\
u(x,t)&=&0,                   \quad  x\in \Om, t<0, \label{wave2}\\
\f{\p u}{\p \nu}(x,t)&=&0,                   \quad   (x,t)\in \Om \times (0,\infty). \label{wave3}
\end{eqnarray}
The solution to the above scattering problem can be written as
$$
u(x, t) = \int_{0}^{\infty}G(x, x_0, t, \tau) f(\tau)d\tau,
$$
where $G(x, x_0, t, t_0)$ is the Green function. More precisely, $G(x, x_0, t, t_0)$ solves (\ref{wave1})-(\ref{wave3})
with $f(t) = \delta (t-t_0)$. It is clear that
$G(x, x_0, t, t_0)= G(x, x_0, t-t_0, 0)$. For simplicity, we also write $G(x, x_0, t)$ for
$G(x, x_0, t, 0)$.

We now formulate the time-reversal experiment. Let $T$ be a sufficiently large positive number
and let $\Sigma$ be a surface where the time-reversal mirrors are distributed. We assume that $\Sigma$ is a surface in the far field region which encloses the objects of interest.

The signals recorded are $$\tilde{s}_1= \tilde{s}_1(y, t) = u(y, t),
\tilde{s}_2= \tilde{s}_2(y, t) = \f{\p u}{\p \nu}(y, t) \quad \mbox{for } t\in [0, T], y\in \Sigma.$$
The recorded data are time-reversed in the following way:
\bea
s_1=s_1(y, t) &=& \tilde{s}_1(y, T-t) ,\\
s_2=s_2(y, t) &=& \tilde{s}_2(y, T-t) .
\eea
These new data are emitted to generate a new field which is given by
\beas
 (TRF)(x, t) &=& \int_{0}^{T} d\tau \int_{\Sigma} d\sigma(y) \left(G(x, y, t, \tau)s_2(y, \tau) - \f{\p G}{\p \nu} (x, y, t, \tau)s_1(y, \tau) \right) \\
  &=& \int_{0}^{T} d\tau \int_{\Sigma} d\sigma(y) \left(G(x, y, t, \tau)\f{\p u}{\p \nu}(y, T-\tau) - \f{\p G}{\p \nu} (x, y, t, \tau)u(y, T-\tau)\right)  \\
&=& \int_{0}^{T} d\tau \int_{\Sigma} d\sigma(y) \left(G(y, x, \tau, T-t)\f{\p u}{\p \nu}(y, \tau) - \f{\p G}{\p \nu} (x, y, \tau, T-t)u(y, \tau)\right).
\eeas
Using integration by parts and second Green's formula, one can derive that
\be
(TRF)(x, t) =
 u(x, T-t) - \int_{0}^T d\tau \;  G(x_0, x, \tau, T-t) f(\tau) + \Theta(x,t),
\ee
where
\beas
\Theta(x,t)&=& \int_{\Omega}d y \left(u_t(y, T) G(y, x, T, T-t)  -\f{\p G}{\p t}(y, x, T, T-t) u(y, T)\right) \\
&=& \int_{\Omega}d y\left(u_t(y, T) G(y, x, t)  -\f{\p G}{\p t}(y, x, t)u(y, T)\right).
\eeas
By the local energy-decaying properties for the wave fields,
$u(x, T)$ and $G(x, y,T)$, we can ignore the reminder term $\Theta(x,t)$ by assuming that $T$ is sufficiently large (a discussion on the estimate of $T$ is given in Appendix \ref{appendixB} for the concrete experiment in Section \ref{sec-sub-TRE}).  Thus
\beas
(TRF)(x, t) & \approx &
 u(x, T-t) - \int_{0}^T  G(x_0, x, \tau, T-t) f(\tau)d\tau \\
 &=& \int_{0}^T \big(  G(x, x_0, T-t, \tau)- G(x_0, x, \tau, T-t)\big) f(\tau)d\tau \\
  &=& \int_{0}^T \bigg(  G( x,x_0, T-t- \tau)- G\big(x_0, x, -(T-t-\tau) \big) \bigg) f(\tau)d\tau.
\eeas
In the special case when $f(t) =\delta(t)$, $u(x, t)= G(x, x_0, t)$. Consequently,
\be
(TRF)(x, t) \approx G(x, x_0, T-t) - G(x_0, x, t-T).
\ee
We are interested in the case when $t \approx T$. We define
\be
{\phi}(x,t) = (TRF)(x, t+T) \approx G(x, x_0, -t) - G(x, x_0, t) = K(x, x_0, t).
\ee
The function $K(x, x_0, t)$ is called the resolution kernel in the time domain.

The time-reversal field for the scattering problem (\ref{wave1})-(\ref{wave3}) with general signal $f$ is given by
\be
{\phi} = {\phi}( x, x_0, t)=\int_0^{T} K(x, x_0,t+ \tau)f(\tau)d\tau.
\ee
We may assume that $f$ is compactly supported in $(0, T)$. Then
\be
{\phi}(x, x_0, t) = \int_0^{\infty} K(x, x_0,t+\tau)f(\tau)d\tau = K(x, x_0, \cdot) \ast f(-\cdot)(t).
\ee
In the Fourier domain, we have
\be
\check{\phi}(x, x_0, \omega) = \check{ K}(x, x_0, \om)\bar{ \check{f}}(\om),
\ee
where
\beas
\check{\phi}(x, x_0, \omega) &=& \int \phi (x, x_0, t) e^{i\om t} dt, \\
\check{K}(x, x_0, \omega) &=& \int K (x, x_0, t) e^{i\om t} dt, \\
\check{f}(\omega) &=& \int f( t) e^{i \om t} dt.
\eeas
Note that $\check{K}(x, x_0, \om)= -\check{G}(x, x_0, \om) + \overline{\check{G}(x, x_0, \om)} = -2 i \Im{\check{G}(x, x_0, \om)}$. Therefore,
\beas
{\phi}(x, x_0, t) &=& \f{1}{2\pi}\int \check{\phi}(x, x_0, \om) e^{-i\om t}d \om \\
&=& -\f{i}{\pi}\int  \Im{\check{G}(x, x_0, \om)} \overline{\check{f}}(\om) e^{-i\om t} d \om.
\eeas
Since for $\om \in \R$, $\check{f}(-\om)= \overline{\check{f}}(\om)$ and
$\Im{\check{G}(x, x_0, \om)} = -\Im \overline{{\check{G}}(x, x_0, \om)}$,
we can further deduce that
\beas
{\phi}(x, x_0, t)&=& -\f{i}{\pi} \left(\int_{-\infty}^{0} \Im{\check{G}(x, x_0, \om)} \overline{\check{f}}(\om)e^{-i\om t}  d \om
+ \int_0^{\infty} \Im{\check{G}(x, x_0, \om)} \overline{\check{f}}(\om)e^{-i\om t}  d \om\right) \\
&=& -\f{2}{\pi} \int_0^{\infty} \Im{\check{G}(x, x_0, \om)} \Im{\big(\check{f}(\om)e^{i\om t}\big)} d \om.
\eeas
In particular, at $t=0$, we get
\beas
{\phi}(x, x_0, 0)=  -\f{2}{\pi} \int_0^{\infty} \Im{\check{G}(x, x_0, \om)} \Im{\check{f}}(\om)  d \om.
\eeas

\subsection{Time-reversal experiments with a system of Helmholtz resonators} \label{sec-sub-TRE}

We present a simplified model of the time-reversal experiment in \cite{LFL2011}. We first introduce the concept of Helmholtz resonator \cite{G1997}.
Let $D= S(0, 1) \times [-h, 0]$, where $S(0, 1)= \{(x_1,x_2) :  x_1^2 + x_2^2 \leq 1\}$ and $h$ is the height of $D$, which is of order one. Let $\Lambda \subset S(0, 1) \subset \R^2$ be a simply connected domain which is of size one and let $\e>0$ be a small number. We assume that $0\in \Lambda$ without loss of generality.
We shall call $D$ or $\partial D \backslash \{(x_1, x_2, 0) :  (x_1, x_2) \in \epsilon \Lambda\}$ a Helmholtz resonator.

We now have a system of such resonators which consists of $M$ disjoint $D_j$'s ($1\leq j \leq M$),
where $D_j= D+ z^{(j)}$ and
$z^{(j)}= (z^{(j)}_1, z^{(j)}_2, 0)$ is the center of aperture for $j$-th resonator.
We denote by $\Om^{in}= \bigcup_{j=1}^{M} D_j$, $\Om^{ex}= \{(x_1, x_2, x_3)\in \R^3 :  x_3 >0\}$ and
$\Om_{\e} = \Om^{in} \bigcup \Om^{ex} \bigcup \Lambda_{\e}$ with $\Lambda_{\e}=\bigcup _{j=1}^M \Lambda_{\e, j}$.

Employing the above setup, the time-reversal experiment is carried out in the domain $\Om= \Om_{\e}$ whereas the time-reversal mirror located at $\Sigma = \{x :  |x|^2 = r^2\} \bigcap \Om^{ex}$ with  $r \gg 1$.

\subsection{Notation and preliminaries}
We first introduce two auxiliary Green's functions.
Let $G^{ex}$ be the Green function for the following exterior scattering problem:
\[ \left\{
\begin{array}{ccc}
(\Delta + k^2)G^{ex}(x, y, k) &=&\delta(x-y),  \quad x \in \Om^{ex}, \\
\f{\p G^{ex}}{\p \nu}(x, y, k) &=& 0, \quad x \in \p \Om^{ex}, \\
G^{ex} \mbox{ satisfies the radiation condition},& &
\end{array} \right.
\]
and $G^{in}$ be the Green function for the following interior problem:
\[ \left\{
\begin{array}{ccc}
(\Delta + k^2)G^{in}(x, y, k) &=& \delta(x-y),  \quad x \in D, \\
\f{\p G^{in}}{\p \nu}(x, y, k) &=& 0, \quad x \in \p D.
\end{array} \right.
\]
Throughout the paper, we denote by
$$
W= \{k\in \mathbb{C}: |k| \leq \f{1}{2} k_1\},
$$
where $k_1$ is the first nonzero eigenvalue of the Neumann problem in $D$.

Note that for ease of notation, we always suppress the script $\check{}$ for  the Green functions in the frequency domain.

We have the following result. The proof is similar as the one in \cite{G1997}.
\begin{lem}
Let $y\in \{x_3=0\}$ and $k \in W$. Then,
\bea
G^{ex}(x, y, k) &=& \f{1}{2\pi |x-y|} +R^{ex}(x, y, k), \quad x\in \Om^{ex},\\
G^{in}(x, y, k) &=& \f{1}{2\pi |x-y|} - \f{\psi(x)\psi(y)}{k^2}   +R^{in}(x, y, k), \quad x\in D,
\eea
where
\beas
R^{ex}(x, y, k)&=& \f{ik}{2\pi}\int_{0}^1 e^{ik|x-y|t} d t, \\
R^{in}(x, y, k)&=& k\int_{0}^1 \sin{ik|x-y|t} d t + r(x, y, k)
\eeas
for some function $r$ which is analytic in $W$ with respect to $k$ and is smooth in a neighborhood of $\Lambda$ in the plane $\{x_3=0\}$ with respect to both the variables $x$ and $y$.
\end{lem}

We denote by
$$
R(x, y, k) = R^{ex}(x, y, k)+R^{in}(x, y, k),
$$
and
\be
\alpha_0= R(0,0,0), \quad \alpha_1= \f{\p R}{\p k}(0, 0, 0).
\ee
It is clear that
\be
\alpha_0 \in \R, \quad \Im \alpha_1 =  \Im \f{\p R^{ex}}{\p k}(0, 0, 0) = \f{1}{2\pi}.
\ee

We now introduce the matrices $T=(T_{ij})_{M\times M}$ and $S=(S_{ij})_{M\times M}$ with
\begin{equation} \label{defTS}
\left\{\begin{array}{lll}
T_{ij}&=& \f{1}{2 \pi |z^{(i)}-z^{(j)}|} \quad \mbox{for }\, i\neq j, \quad \mbox{and }\,\, T_{ii}=0,\\
S_{ij}&=& \f{\sqrt{-1}}{2 \pi} + \delta_{ij} \Re \alpha_1.
\end{array} \right. \end{equation}

Observe that $T$ is symmetric, thus $T$ has $M$ real eigenvalues, which are denoted by $\beta_1, \beta_2,\ldots, \beta_M$.
For the ease of exposition, we make in the sequel the following assumption.
\begin{asump} We assume that
$ \beta_1,\ldots,\beta_M \mbox{ are mutually distinct}.$
\end{asump}

This is the generic case among all the possible arrangements of the resonators.
The corresponding normalized eigenvectors are denoted by $Y_1, Y_2,\ldots,Y_M$, respectively. Then
$Y_1, Y_2, \ldots, Y_M$ form a normal basis for $\R^{M}$. We also denote by $Y$ the matrix
$$
Y= (Y_1, Y_2,\ldots,Y_M).
$$

For convenience, we write
\be
\mathcal{G}(x, k) = \big(G^{ex}(x, z^{(1)}, k), G^{ex}(x, z^{(2)}, k),\ldots, G^{ex}(x, z^{(M)},k)\big)^t,
\ee
with the subscript $t$ denoting the transpose.
For each $1\leq j \leq M$, we denote by
\begin{equation} \label{defzeta}
\zeta_j(x, x_0, k) = \mathcal{G}(x, k)^t Y_j Y_j^t \mathcal{G}(x_0, k).
\end{equation}
It is clear that $\zeta_j=\zeta_j(x, x_0, k)$ is analytic in $k$ for fixed $x$ and $x_0$.

\subsection{Main results on the resonances of a system of Helmholtz resonators and the Green function}

It is evident from Section \ref{sec-sub-TR} that the focusing property of the time-reversal experiment in Section \ref{sec-sub-TRE} relies in analysis of the following Green function in the frequency domain
\[ \left\{
\begin{array}{l}
(\Delta + k^2) {G}_{\e}(x, x_0, k) =\delta(x-x_0),  \quad x \in \Om_{\e}, \\
\nm
\f{\p {G}_{\e}}{\p \nu}(x, x_0, k) = 0, \quad x \in \p \Om_{\e}, \\
 \nm
 {G}_{\e}  \mbox{ satisfies the radiation condition}.  \label{helm3-3}
\end{array}
\right.
\]

We first present a result on the resonances of the above scattering problem. The proof is given in Section \ref{sec-sub-resonance-multi}.

\begin{prop} \label{prop-resonances}
 There exist exactly 2M resonances of order one in the domain $W$ for the system of resonators in Section \ref{sec-sub-TRE}, given by
 \bea
 k_{0, \e, j, 1} &=&  \tau_1 \e^{\f{1}{2}} + \tau_{3,j}  \e^{\f{3}{2}} + \tau_{4,j} \e^2 +O(\e^{\f{5}{2}}), \label{formula-k1-1}\\
 k_{0, \e, j, 2} &=& - \tau_1 \e^{\f{1}{2}} - \tau_{3,j} \e^{\f{3}{2}}
+ \tau_{4,j} \e^2+O(\e^{\f{5}{2}}), \label{formula-k1-2}
\eea
where
\begin{equation} \label{deftau1}
\tau_1= \sqrt{\f{ c_{\Lambda}}{|D|}},
\end{equation}
\begin{equation} \label{deftau3}
\tau_{3, j}= -\f{1}{2}(\alpha_0+ \beta_j )\left(\f{ c_{\Lambda}}{|D|}\right)^{\f{1}{2}}c_{\Lambda},
\end{equation}
and
\begin{equation} \label{deftau4}
\tau_{4, j}= - \f{1}{2}\f{c_{\Lambda}^2}{|D|} Y_j^t SY_j
\end{equation}
with $c_{\Lambda}$ being the capacity of the set $\Lambda$ defined by (\ref{constant-capcity}).
\end{prop}

\begin{rmk}
For each nonzero eigenvalue $k_n$ of the Neumann problem in $D$, there exist $2M$ resonances, counting multiplicity, in a neighborhood of $k_n$ in the lower-half complex space. The method developed in the paper can be used to derive the leading asymptotic of the resonances.
\end{rmk}

\begin{rmk}
The approach developed in the paper can be used to derive the full asymptotic of the resonances. We derive only the leading four terms here which suffices for our purpose.
\end{rmk}

We now state our main result on the Green function in $\Om_{\e}$. See Section \ref{sec-sub-green-multi} for the proof.
\begin{thm} \label{thm-multi}
Assume that $k \in \R \bigcap W$. Then for $\e$ sufficiently small, the exterior Green's function has the following asymptotic expansion
\beas
G_{\e}^{ex}(x, x_0, k) &=& G^{ex}(x, x_0, k) - {\e}c_{\Lambda}\sum_{1\leq j \leq M}G^{ex}(z^{(j)}, x_0, k)G^{ex}(x, z^{(j)}, k)\\
&& - \sum_{j=1}^M \left( \f{1}{k-k_{0,\e,j,2}} -\f{1}{k-k_{0,\e,j,1}}\right)\f{(c_{\Lambda}\e)^{\f{3}{2}}}{\sqrt{|D|}} \mathcal{G}(x, k)^tY_j Y_j^t \mathcal{G}(x_0, k) \\
&& + \sum_{1\leq j \leq M } \left(\f{O(\e^{2})}{k- k_{0,\e,j, 2} }+ \f{O(\e^{2})}{k- k_{0,\e,j, 1} }\right) + O(\e^{2}).
\eeas
\end{thm}

\begin{rmk}
The method developed in the paper can be used to derive similar results about the Green function in the whole space with the presence of a system of Helmholtz resonators as well as other settings. The advantage of the setting in Section \ref{sec-sub-TRE} is that there is no resonance for the limiting exterior scattering problem, which is not the case for a general setting. However, the method in the paper still applies and one needs only to shrink the region $W$ to exclude the resonances from the limiting exterior scattering problem.

The method can be also used to derive asymptotic of the Green function when the frequency is close to any of the nonzero eigenvalues of the Neumann problem in $D$.
\end{rmk}

\subsection{Main results on the super-resolution (or super-focusing)} \label{sec-intro-sub-super-resolution}

As a consequence of the result of the Green function in the previous section, we
can establish the following result on super-resolution (or super-focusing), which
shows that super-resolution can be achieved with a single specific frequency. See Section \ref{sec-sub-proof-super-1} for a detailed proof.

\begin{thm} \label{thm-super-resolution-1}
Let $\tau_1$ be given by (\ref{deftau1}), where $c_{\Lambda}$ is the capacity of the set $\Lambda$ defined by (\ref{constant-capcity}). For $k= \tau_1 \sqrt{\e}$, the resolution function $\Im{G}_{\e}^{ex}$ has the following estimate:
$$
\Im{G}_{\e}^{ex}(x, x_0, k)= \f{\sin{\tau_1 \sqrt{\e}|x-x_0|}}{2\pi|x-x_0|}+\f{c_{\Lambda}^{\f{3}{2}}}{|D|^{\f{1}{2}}} \e^{\f{1}{2}}\sum_{j=1}^M
\frac{\Im \tau_{4,j}}{\tau^2_{3,j}}
 \zeta_j(x, x_0,0) + O(\e),
$$
where $\zeta_j(x, x_0, 0)$ is given by (\ref{defzeta}) and $\tau_{3,j}$ and $\tau_{4,j}$ are defined by (\ref{deftau3}) and (\ref{deftau4}), respectively.
\end{thm}

We now consider the case when the signal is broadband. We shall prove that super-resolution can be achieved when the signal is concentrated in the quasi-stationary regime. A signal $f$ is said to be quasi-stationary regime if $\check{f}(\cdot)$ is mainly supported in some $O(\sqrt{\e})$ neighborhood of the origin (a precise definition will be given later).
The result is based on the study of the following imaging functional:
$$
I = \int_0^{\infty} \Im{{G}_{\e}^{ex}(x, y_0, k)} \Im{\big(\check{f}(k)e^{ikt}\big)} d k.
$$

We write $f(t)$ in the following form
\be
f(t) = \e^{\f{1}{4}} F(\e^{\f{1}{2}}t),
\ee
where $F$ is the root signal such that the following two natural conditions holds
\bea
& \mbox{supp }F \subset [0, C_1],  \label{cond-supp}\\
& \int_{0}^{\infty} F(t)^2 dt = O(1) \nonumber.
\eea


\begin{definition}
We say that the signal $f= f(t)= \e^{\f{1}{4}} F(\e^{\f{1}{2}}t)$ is quasi-stationary if, together with (\ref{cond-supp}), the following conditions
are satisfied
\bea
\|F\|_{H^2(\R)} &=& O(1), \label{cond-H2} \\
\int_{\e^{- \delta}}^{\infty} |\check{F}(k)| dk  & \ll & O(\e) \quad \mbox{for some $\delta \in (0, \f{1}{2})$}, \label{cond-quasi1} \\
\int_{\f{k_1}{2\sqrt{\e}}}^{\infty}  \Im{G}_{\e}^{ex}(x, x_0, \sqrt{\e}k)\Im{\big( \check{F}(\tau_1)e^{-i\tau_1 \sqrt{\e}t}\big)} dk & \ll & O( \e). \label{cond-quasi2}
\eea
\end{definition}

Discussions on the condition (\ref{cond-quasi1}) and (\ref{cond-quasi2}) are given in Section \ref{sec-sub-proof-super-2}, after the proof of Theorem \ref{thm-super-resolution-2}.

We are now ready to state our main theorem on the super-resolution. The proof is given in Section \ref{sec-sub-proof-super-2}.
\begin{thm} \label{thm-super-resolution-2}
For the time-reversal experiment in Section \ref{sec-sub-TRE}, super-resolution (super-focusing) can be claimed when
the signal $f= \e^{\f{1}{4}} F(\e^{\f{1}{2}}t)$ is quasi-stationary. Moreover, the imaging functional has the following form:
\beas
I &=&  \int_0^{2 \tau_1 \e^{\f{1}{2}}} \f{\sin{k|x-x_0|}}{2\pi|x-x_0|} \e^{-\f{1}{4}} \Im{\big( \check{F}(\e^{-\f{1}{2}}k)e^{ikt}\big)} dk \\
 &&  +\f{(c_{\Lambda})^{\f{3}{2}}}{\sqrt{|D|}}  \e^{\f{5}{4}}  \Im{\big( \check{F}(\tau_1)e^{i\tau_1 \sqrt{\e}t}\big)} \sum_{j=1}^M \f{1}{4\pi |x- z^{(j)}|\cdot |x_0- z^{(j)}| }  + o( \e^{\f{5}{4}}).
\eeas
\end{thm}

As will be shown later, the super-resolution is due to the term
$$
\f{(c_{\Lambda})^{\f{3}{2}}}{\sqrt{|D|}}  \e^{\f{5}{4}}  \Im{\big( \check{F}(\tau_1)e^{i\tau_1 \sqrt{\e}t}\big)} \sum_{j=1}^M \f{1}{4\pi |x- z^{(j)}|\cdot |x_0- z^{(j)}| } .
$$
This term allows to find the location $x_0$ of the source within a resolution of order $O(1)$ while the first term in $I$ yields a resolution limit (or a focal spot of size) of order $O(\e^{\delta-\f{1}{2}})$.

\begin{rmk}
The term ``sub-wavelength resonator" is associated with scattering in the quasi-stationary regime. In fact, in the case of the Helmholtz resonator, it is in that regime that the free space wavelength is significantly greater than the size of the resonator. We also remark that the resonance in the quasi-stationary regime results from the perturbations  of the zero-eigenvalue of the Neumann problem in the closed resonator that are due to small openings.

\end{rmk}

\section{Preliminaries on potential theory}  \label{sec-potential}
Let $\Lambda \subset \{(x_1, x_2, x_3) \in \R^3: x_3=0\}$ be a simple connected surface with smooth boundary.
Let $\mu \in \mathcal{S}'(\R^2)$ whose support is contained in $\bar{\Lambda}$. Denote by $\hat{\mu}=\hat{\mu}(\xi)$ its Fourier transform. We then define the energy of $\mu$ by the following formula
$$
\|\mu\|^2 = \frac{1}{2\pi^2} \int_{\R^2}\f{|\hat{\mu}(\xi)|^2}{|\xi|}d\xi.
$$
The space of above $\mu$ with finite energy is denoted by $V$, i.e.,
$$
V= \{\mu: \mbox{supp} \mu \subset \bar{\Lambda}, \int_{\R^2}\f{|\hat{\mu}(\xi)|^2}{|\xi|} d\,\xi < \infty\}.
$$
It is clear that $V$ is a Hilbert space. We now define the Riesz potential on the space $V$; see \cite{L1972}. For each $\mu\in V$, we define
\be
 \mathcal{L}[\mu] (x) = \f{1}{\pi} (\mu, \f{1}{|x-\cdot|})_{\Lambda}=\int_{\Lambda} \f{\mu(y)}{\pi |x-y|}dy, \quad x\in \R^3 \backslash \Lambda.
\ee

Let $U(x)=\mathcal{L}[\mu] (x)$, then one can show from \cite[Chapter VI]{L1972} that $U\in H^1(\R^3)$ and $U$ satisfies the following equation
$$
- \triangle U = 4 \delta(x_3)\mu.
$$
Moreover, the following identities hold
\be \label{potential-identity}
(\mu, U)_{\Lambda} =  \int_{\Lambda\times \Lambda} \f{\mu(x)\mu(y)}{\pi |x-y|}dxdy  = \frac{1}{4} \int_{\R^3} |\nabla U|^2 dx =\|\mu\|^2.
\ee
Denote by $\kappa$ and $\kappa_1$ the trace operator from $\R^3$ to $\R^2=\{x_3=0\}$ and $\Lambda$, respectively. For each $f\in H^1(\R^3)$, $\kappa(f) = f|_{\{x_3=0\}} \in H^{\f{1}{2}}(\R^2)$ and $\kappa_1(f)= f|_{\Lambda}$.
We have
\beas
\kappa(U)(x_1, x_2) &=& \frac{1}{(2\pi)^3} \int_{\R^3} e^{i (x_1\xi_1+ x_2\xi_2)} \hat{U}(\xi_1, \xi_2, \xi_3) d\xi_1d\xi_2d\xi_3 \\
 &=& \frac{1}{2 \pi^3} \int_{\R^3} e^{i (x_1\xi_1+ x_2\xi_2)} \f{\hat{\mu}(\xi_1, \xi_2)}{\xi_1^2 + \xi_2^2+  \xi_3^2} d\xi_1d\xi_2d\xi_3 \\
 &=&  \frac{1}{2 \pi^2} \int_{\R^2} e^{i (x_1\xi_1+ x_2\xi_2)} \f{\hat{\mu}(\xi_1, \xi_2)}{(\xi_1^2 + \xi_2^2)^{\f{1}{2}}} d\xi_1d\xi_2.
\eeas

It follows that in the Fourier space, the operator $\kappa\circ \mathcal{L}$ becomes a multiplier. In fact, we have
$$
\widehat{\kappa\circ \mathcal{L}[\mu]} (\xi_1, \xi_2) = \f{ 2 \hat{\mu}(\xi_1, \xi_2)}{(\xi_1^2+ \xi_2^2)^{\f{1}{2}}}.
$$

We denote by $V^*$ the dual space of $V$.  We establish the following result.
\begin{lem}
 The linear operator $\kappa_1\circ \mathcal{L}$ is an isometry between the space $V$ and $V^*$.
\end{lem}
\begin{proof}
We first show that
$$
\kappa_1\circ \mathcal{L}[V] \subset V^*.
$$
Indeed, for any $\nu \in V$, we have
\beas
|(\kappa_1\circ \mathcal{L}[\mu], \nu )|&=& |\int_{\Lambda\times \Lambda} \f{\mu(x)\nu(y)}{\pi |x-y|}dxdy|  \\
& \leq &
\left(\int_{\Lambda\times \Lambda} \f{\mu(x)\mu(y)}{\pi |x-y|}dxdy \right)^{\f{1}{2}}\cdot
\left(\int_{\Lambda\times \Lambda} \f{\nu(x)\nu(y)}{\pi |x-y|}dxdy \right)^{\f{1}{2}} \\
&=& \|\mu\| \cdot \|\nu\|.
\eeas
This proves that $\kappa_1\circ \mathcal{L}[\mu] \in V^*$ and $\|\kappa_1\circ \mathcal{L}[\mu]\|_{V^*} \leq \|\mu\|$. In addition, by using the identities in (\ref{potential-identity}), we see that
$$
\|\kappa_1\circ \mathcal{L}[\mu] \|_{V^*} = \|\mu\|.
$$
Thus the operator $\kappa_1\circ \mathcal{L}$ is an isometric embedding of $V$ into $V^*$, which also yields that
$\kappa_1\circ \mathcal{L}[V]$ is a closed subspace of $V^*$.
We finally show by contradiction that $\kappa_1\circ \mathcal{L}[V] = V^*$.
Assume the contrary, then there exists a nontrivial $\mu \in V$ such that
$$
( \kappa_1\circ \mathcal{L}[\nu],  \mu) =0 \quad \mbox{for all $\nu \in V$}.
$$
By taking $\nu=\mu$, we obtain
$$
( \kappa_1\circ \mathcal{L}[\mu],  \mu) = \|\mu\|^2 =0,
$$
which implies that $\mu=0$. This contradiction proves that $\kappa_1\circ \mathcal{L}[V] = V^*$. The lemma is proved.\end{proof}

\medskip

It is clear that $\kappa_1 (H^{\f{1}{2}}(\R^2)) \subset V^*$. As a consequence of the above lemma, we obtain the solvability result for following problem.
\begin{lem}
Let $f \in H^{\f{1}{2}}(\R^2)$. Then there exists an unique solution to the following problem
\[
\left\{\begin{array}{ccc}
 \Delta u &=&0  \quad \mbox{in }\,\, \R^3\backslash \Lambda, \\
   u &=&f    \quad \mbox{on }\,\, \Lambda, \\
   u(x) &\rightarrow& 0   \quad \mbox{as} \,\,|x| \rightarrow \infty.
\end{array} \right.
\]
Moreover, the solution $u$ can be written as
$u(x)= \mathcal{L}[\mu](x)$ for a uniquely determined $\mu \in V$ satisfying $\|\mu\| \leq \|f\|_{H^{\f{1}{2}}(\R^2)}$.
\end{lem}

Finally, we define the capacity of the set $\Lambda$ by
\be  \label{constant-capcity}
c_{\Lambda} = (\mathcal{L}^{-1}[1], 1).
\ee

In the case where $\Lambda=\{ x= (x_1,x_2,0): |x| \leq 1\}$, we have
$$\mathcal{L}^{-1}[1] = \left\{ \begin{array}{l} \ds  \frac{1}{\pi} (1- |x|^2)^{-1/2}, \quad |x| \leq 1,\\
\nm 0, \quad |x| >1,\end{array} \right.$$ and hence, $c_\Lambda =2$.

\begin{rmk}
Note that the definition of the capacity given by (\ref{constant-capcity}) differs from the standard one by a multiplicative constant. We use here (\ref{constant-capcity}) for the ease of presentation.
\end{rmk}

\section{A single resonator problem} \label{sec-resonator-single}

\subsection{Introduction}

Let $D$, $\Lambda$ and $\Om^{ex}$ be as in Section \ref{sec-sub-TRE}. Denote by
$\Om_{\e} = D \bigcup \Om^{ex} \bigcup \e \Lambda$.
We aim at finding the resonances and the Green function for the following problem in $\Om_{\e}$:
\bea
(\Delta + k^2)G_{\e}(x, x_0, k) =\delta(x-x_0),  \quad x \in \Om_{\e},  \label{helm1}\\
\nm
\f{\p G_{\e}}{\p \nu}(x, x_0, k) = 0, \quad x \in \left(\p \Om^{ex} \bigcup \p D \right)\backslash \Lambda_{\e}, \label{helm2}\\
\nm
 G_{\e} \mbox{ satisfies the radiation condition}.
 \label{helm3}
\eea

Denote by $\varphi_{\e}(x)= -\f{\p G_{\e}}{\p x_3}(\cdot, x_0, k)|_{\Lambda_{\e}}$,
$G_{\e}|_{\Om^{ex}} = G^{ex}_{\e}$ and $G_{\e}|_{D} = G^{in}_{\e}$. Then $G^{ex}_{\e}$ is the solution to the following
exterior problem
\[
\left\{\begin{array}{ccc}
(\Delta + k^2)G^{ex}_{\e}(x, x_0, k) &=&\delta(x-x_0),  \quad x \in \Om^{ex}, \\
\nm
\f{\p G^{ex}_{\e}}{\p \nu}(x, x_0, k) &=& 0, \quad x \in \p \Om^{ex}\backslash \Lambda_{\e}, \\
\nm
\f{\p G^{ex}_{\e}}{\p \nu}(x, x_0, k) &=& \varphi_{\e} , \quad x \in \Lambda_{\e},
\end{array} \right.
\]
while $G_{\e}^{in}$ is the solution to the following interior problem
\[
\left\{\begin{array}{ccc}
(\Delta + k^2)G^{in}_{\e}(x, x_0, k) &=&0,  \quad x \in D, \\
\nm
\f{\p G^{in}_{\e}}{\p \nu}(x, x_0, k) &=& 0, \quad x \in \p D \backslash \Lambda_{\e}, \\
\nm
\f{\p G^{in}_{\e}}{\p \nu}(x, x_0, k) &=& -\varphi_{\e}, \quad  x\in \Lambda_{\e}.
\end{array} \right.
\]

We can decompose $G^{ex}_{\e}$ as two parts: $G^{ex}$ and $F^{ex}_{\e}$. It is easy to see that
\[
\left\{\begin{array}{ccc}
(\Delta + k^2)F^{ex}_{\e}(x, x_0, k) &=&0,  \quad x \in \Om^{ex} ,\\
\nm
\f{\p F^{ex}_{\e}}{\p \nu}(x, x_0, k) &=& 0, \quad x \in \p \Om^{ex}\backslash \Lambda_{\e}, \\
\nm
\f{\p F^{ex}_{\e}}{\p \nu}(x, x_0, k) &=& \varphi_{\e}, \quad x \in \Lambda_{\e}.
\end{array} \right.
\]
Hence,
$$
F^{ex}_{\e}(x, x_0, k) = \int_{\Lambda_{\e}} G^{ex}(x, y, k) \varphi_{\e}(y)dy, \quad x\in \Om^{ex}.
$$
Similarly, we have
$$
G^{in}_{\e}(x, x_0, k) = -\int_{\Lambda_{\e}} G^{in}(x, y, k) \varphi_{\e}(y)dy, \quad x\in D.
$$
Thus, $\varphi_{\e}$ satisfies the following equation
\be
\int_{\Lambda_{\e}} \big(G^{ex}(x, y, k)+ G^{in}(x, x_0, k)\big)\varphi_{\e}(y)dy + G^{ex}(x, x_0, k)=0.
\ee
Recall that
$$
G^{ex}(x, y, k)+ G^{in}(x, x_0, k) = \f{1}{\pi|x-y|} - \f{\psi(x)\psi(y)}{k^2} +R(x, y, k).
$$
We introduce the following three integral operators:
\bea
\mathcal{L}_{\e}[f](x) &=& \int_{\Lambda_{\e}} \f{1}{\pi|x-y|} f(y)dy, \\
\mathcal{K}_{\e}[f](x) &=& \left(\int_{\Lambda_{\e}} \psi (y) f(y)dy\right) \psi(x) = \f{1}{|D|}\int_{\Lambda_{\e}} f(y)dy, \\
\mathcal{R}_{\e}[f](x) &=& \int_{\Lambda_{\e}} R(x, y, k) f(y)dy.
\eea

\begin{lem}
The perturbed Green function $G_{\e}(x, x_0, k)$ to the scattering problem (\ref{helm1}-\ref{helm3}) has the following representation
\bea
G^{ex}_{\e}(x, x_0, k) &=&  G^{ex}(x, x_0, k) + \int_{\Lambda_{\e}} G^{ex}(x, y, k) \varphi_{\e}(y)dy, \quad x\in \Om^{ex}, \label{green-exterior} \\
G^{in}_{\e}(x, x_0, k) &=& \int_{\Lambda_{\e}} G^{in}(x, y, k) \varphi_{\e}(y)dy, \quad x\in D,
\eea
where the unknown function $\varphi_{\e}(y)$ satisfies the following integral equation
\be  \label{BIE-green function}
\big(\mathcal{L}_{\e} - \f{\mathcal{K}_{\e}}{k^2} + \mathcal{R}_{\e}\big)[\varphi_{\e}](x) = -G^{ex}(x, x_0, k), \quad x\in \Lambda_{\e}.
\ee
\end{lem}

We shall prove the existence and uniqueness of the solution to the integral equation (\ref{BIE-green function}) in the subsequent sections.

\subsection{Properties of the operator $\mathcal{L}_{\e}$}

Denote by $V_{\e}$ the space of distributions whose supports are contained in $\bar{\Lambda}_{\e}$ and whose energy is finite. To facilitate the analysis of the aforementioned operators, we introduce two scaling operators.

For each $\mu \in V_{\e}$, we define $\iota_{\e,1} [{\mu}](y)= \e \mu (\e y) \in V_{1}$.
For each $U \in V_{\e}^*$, we define $\iota_{\e, 2} [{U}](y)=  U(\e y) \in V_{1}$.

\begin{lem}
The following identities hold:
\bea
 \|\mu\|_{V_{\e}}  &=& \sqrt{\e} \|\iota_{\e,1} [{\mu}]\|_{V_{1}}, \label{identity1} \\
(\mu, U)_{\Lambda_{\e}}  &=& \e \big(\iota_{\e,1} [{\mu}], \iota_{\e, 2} [{U}] \big)_{\Lambda},\label{identity2} \\
\|U\|_{V_{\e}^*} &=& \sqrt{\e}\|\iota_{\e, 2} [{U}] \|_{V_1^*}, \label{identity3}\\
\mathcal{L}_{\e}^{-1} [U] (y) &=& \f{1}{\e} \mathcal{L}_1^{-1} \iota_{\e, 2} [{U}](\f{y}{\e}) =\iota_{\e,1}^{-1}\mathcal{L}_1^{-1} \iota_{\e, 2} [{U}].  \label{scaling}
\eea
\end{lem}
\begin{proof} We first show (\ref{identity1}).
Let $\mu_1=\iota_{\e,1} [{\mu}] $.
By direct calculation, we have
$$
\hat{\mu}_1 (\xi) = \f{1}{\e} \hat{\mu}(\f{\xi}{\e}).
$$
It follows that
$$
\|\mu_1\|^2_{V_1}
=\int_{\R^2} \f{|\hat{\mu}_1(\xi)|^2}{|\xi|} d\xi
=\int_{\R^2} \f{|\hat{\mu}(\xi)|^2}{\e |\xi|} d\xi = \f{1}{\e} \|\mu\|_{V_{\e}}^2.
$$
This proves (\ref{identity1}).

The identity (\ref{identity2}) follows from a standard change of variables argument.
We now show (\ref{identity3}).
In fact, we have
\[
\begin{array}{ll}
\|U\|_{V_{\e}^*} & = \sup_{\mu: \|\mu\|\leq 1}(\mu, U)
= \sup_{\mu: \|\mu\|\leq 1}\e \big(\iota_{\e,1} [{\mu}], \iota_{\e, 2}[U]\big) \\
&=\sup_{ \tilde{\mu}: \|\tilde{\mu}\|_{V_1} \leq 1}\e \big(\tilde{\mu},  \iota_{\e, 2}[U] \big)_{\Lambda}\\
&= \sqrt{\e}\|\iota_{\e, 2} [{U}]\|_{V_1^*}.
\end{array}
\]
Finally, (\ref{scaling}) follows from the observation that if $U$ solves the equation $-\Delta U = 4\delta(x_3)\mu$, then $\iota_{\e, 2} [{U}]$ solves the equation
$-\Delta \iota_{\e, 2} [{U}] = 4\delta(x_3)\iota_{\e,1} [{\mu}]$.

This completes the proof of the lemma.\end{proof}

\medskip

\begin{lem} \label{lem-estimate-R}
The following estimate holds for the operator $\mathcal{L}_{\e}^{-1}\mathcal{R}_{\e}$:
 \be
 \|\mathcal{L}_{\e}^{-1}\mathcal{R}_{\e} \|_{\textbf{L}(V_{\e} \rightarrow V_{\e})}  \lesssim \e.
 \ee
\end{lem}
\begin{proof} Define $\tilde{R}(x, y, k)= R(\e x, \e y, k)$ and the corresponding integral operator by $\mathcal{\tilde{R}}_{\e}$. By the representation of the function $R$, we can show that the operator
$\mathcal{\tilde{R}}_{\e}$ is bounded from $V_1$ to $V_1^*$.
Note that
\beas
\iota_{\e, 2} \mathcal{R}_{\e}[\mu] (x)&=& \mathcal{R}_{\e}[\mu] (\e x) = \int_{\Lambda_{\e}} R(\e x, y, k) \mu(y)dy  \\
&=&\int_{\Lambda} R(\e x, \e y, k) \e \iota_{\e, 1}[\mu](y)dy  \\
&=& \e \mathcal{\tilde{R}}_{\e}[\iota_{\e, 1}[\mu]](x).
\eeas
Consequently,
\beas
\mathcal{L}_{\e}^{-1}\mathcal{R}_{\e}[\mu](y) &=& \f{1}{\e} \mathcal{L}_1^{-1} [ \e \mathcal{\tilde{R}_{\e}}[\iota_{\e, 1}[\mu]] ](\f{y}{\e})  \\
&=&  \mathcal{L}_1^{-1} [\mathcal{\tilde{R}_{\e}}[\iota_{\e, 1}[\mu]] ](\f{y}{\e}).
\eeas
It follows that
$$
\iota_{\e, 1}\circ \mathcal{L}_1^{-1}\circ \mathcal{R}_{\e}[\mu] = \e \mathcal{L}_1^{-1} \circ
\mathcal{\tilde{R}_{\e}}[\iota_{\e, 1}[\mu]].
$$
Thus,
\beas
\|\mathcal{L}_{\e}^{-1}\mathcal{R}_{\e}[\mu] \|_{V_{\e}}
&=& \e^{\f{3}{2}} \| \mathcal{L}_1^{-1} \circ
\mathcal{\tilde{R}_{\e}}[\iota_{\e, 1}[\mu]] \| \\
&\lesssim & \e^{\f{3}{2}} \|\mathcal{L}^{-1}_1\|_{\textbf{L}(V_1^* \rightarrow V_1)} \cdot \| \mathcal{\tilde{R}}_{\e}\|_{\textbf{L}(V_1 \rightarrow V_1^*)} \cdot \| \iota _{\e, 1}[\mu] \|_{V_1} \\
&\lesssim & \e \|\mu\|_{V_{\e}}.
\eeas

%
%
%
This completes the proof of the lemma.\end{proof}

\medskip

As a consequence of the above lemma, we have the following result.
\begin{lem}
 The operator $Id + \mathcal{L}_{\e}^{-1}\mathcal{R}_{\e}$ is invertible from $V_{\e}$ to $V_{\e}$. Moreover,
 $$
 (Id + \mathcal{L}_{\e}^{-1}\mathcal{R}_{\e})^{-1} = \sum_{n=0}^{\infty} (-\mathcal{L}_{\e}^{-1}\mathcal{R}_{\e})^n.
 $$
\end{lem}

\subsection{Existence and asymptotic of the quasi-static resonances}

We aim to show the existence of characteristic values for the operator-valued function
\be  \label{resonance-equation1}
 \mathcal{A}_{\e}(k)[\mu] = \big(\mathcal{L}_{\e} - \f{\mathcal{K}_{\e}}{k^2} + \mathcal{R}_{\e}(k)\big)[\mu] =0
\ee
and determine its asymptotic.

We first look at a simpler problem. We define $\mathcal{J}_{\e} = \mathcal{L}_{\e} - \f{\mathcal{K}_{\e}}{k^2}$ and
$k_{0. \e, 0} = \sqrt{\f{1}{|D|} \e c_{\Lambda}}$. We investigate the characteristic values and the associated characteristic functions for the operator-valued analytic function $\mathcal{J}_{\e}$.

\begin{lem}
The operator $\mathcal{J}(k)$ has two characteristic values which are $\lambda_1= k_{0. \e, 0}$,
$\lambda_2= -k_{0. \e, 0}$. The associated characteristic functions are
$\mu_{1} =\mu_{2} =\mathcal{L}_{\e}^{-1}[\psi]{k_{0,\e,0}^2}$ after imposing the normalization condition
$(\mu, \psi) =1$.
\end{lem}
\begin{proof}

Let $\mu \neq 0$ and $k$ be such that
$$
\mathcal{J}_{\e}(k)[\mu] = (\mathcal{L}_{\e} - \f{\mathcal{K}_{\e}}{k^2})[\mu] =\mathcal{L}_{\e} [\mu] - \f{(\mu, \psi) \psi}{k^2}=0.
$$
Then $\mu - \f{(\mu, \psi) \mathcal{L}_{\e}^{-1} [\psi] }{k^2} = 0 $.
Apply both side by $\psi$, we get
$$
(\mu, \psi) (1- \f{(\mathcal{L}_{\e}^{-1} [\psi], \psi) }{k^2}) =0.
$$
It is clear that $(\mu, \psi)\neq 0$ for otherwise it would imply that $\mathcal{L}_{\e}[\mu]=0$ which further implies that $\mu=0$. Thus, we have
$$
k^2 = (\mathcal{L}_{\e}^{-1}[\psi], \psi) = \f{1}{|D|} (\mathcal{L}_{\e}^{-1}[1], 1) =  \f{1}{|D|} c_{\Lambda_{\e}} =  \f{1}{|D|} \e c_{\Lambda}.
$$

Therefore, $\pm k_{0. \e,0}$ are the characteristic values to $\mathcal{J}_{\e}(k)$.  We now find the corresponding characteristic functions.
Recall the identity
$$
\mu - \f{(\mu, \psi) \mathcal{L}_{\e}^{-1}[\psi] }{k_{0,\e,0}^2} = 0
$$
by the normalization condition $(\mu, \psi) =1$, we obtain immediately the solution
$$
\mu = \mathcal{L}_{\e}^{-1}[\psi]{k_{0,\e,0}^2}.
$$
This completes the proof of the lemma. \end{proof}

\medskip

We now consider the resonance problem (\ref{resonance-equation1}).
Denote by $\mathcal{L}_{\e, 1} = \mathcal{L}_{\e} + \mathcal{R}_{\e}$ and
$\mathcal{B}=\mathcal{B}(k)= \mathcal{L}_{\e}^{-1}\mathcal{R}(k)$. Recall that
$\| \mathcal{B}\| \lesssim \e$.
For $\e$ sufficiently small, $\mathcal{L}_{\e, 1}$ is invertible and
\be
\mathcal{L}_{\e, 1}^{-1} = (Id + \mathcal{B})^{-1} \mathcal{L}_{\e}^{-1}
=\sum_{n=0}^{\infty}(-\mathcal{B})^n \mathcal{L}_{\e}^{-1}.
\ee
Note that $\mu \neq 0$ and $k$ satisfy (\ref{resonance-equation1}) if and only if
\be
\mu -  \f{(\mu, \psi) \mathcal{L}_{\e, 1}^{-1} [\psi] }{k^2}=0.
\ee
Apply both sides by $\psi$ and use the similar argument as before, we derive that
\be  \label{formula-a(k)}
 k^2 = \left(\mathcal{L}_{\e, 1}^{-1} [\psi], \psi\right)=\sum_{n=0}^{\infty}\left((-\mathcal{B}(k))^n \mathcal{L}_{\e}^{-1}[\psi], \psi\right).
\ee
Denote by
$$
A(k, \e) = \left(\mathcal{L}_{\e, 1}^{-1}[\psi], \psi\right)= \left(\mathcal{L}_{\e, 1}^{-1}(k) [\psi], \psi\right).
$$
Then $A(k, \e)$ is analytic in $k$ and is smooth in $\e$. We have established the following result.
\begin{lem} \label{lem-connect}
The characteristic values of the operator-valued analytic function $\mathcal{A}_{\e}(k)$ (or resonances) are zeros of the function (with variable $k$)
$$
k^2- A(k, \e) =0.
$$
\end{lem}

It is clear from (\ref{formula-a(k)}) that the characteristic values (or resonances) satisfy the following equation:
$$
k^2 = (\mathcal{L}_{\e}^{-1}[\psi], \psi) + O(\e^2).
$$
We can also derive that the corresponding characteristic functions are
$$
\mu = \f{\mathcal{L}_{\e}^{-1}[\psi]}{(\mathcal{L}_{\e}^{-1}[\psi], \psi)} + O(\e).
$$

We now use the generalized Rouche's theorem to deduce rigorously the existence of characteristics values for
$ \mathcal{A}_{\e}(k)$.

\begin{lem} \label{lem-characteristic1}
 There exist two characteristic values, counting multiplicity, for the operator-value analytic function $\mathcal{A}_{\e}(k)$ in $W=\{ k\in \mathbb{C} : |k|\leq \f{1}{2} k_1\}$.
 Moreover, they have the asymptotic
 $$
 k_{0, \e} = \pm k_{0, \e , 0} + O(\e).
 $$
\end{lem}
\begin{proof}
Recall that the operator-valued analytic function $\mathcal{J}_{\e}(k)$ is finite meromorphic and of Fredholm type. Moreover,
it has two characteristic values $\pm k_{0, \e, 0}$, and has a pole $0$ with order two in $W$. Thus, the multiplicity of $\mathcal{J}_{\e}(k)$ is zero in $W$. Note that for $k \in W \backslash \{0, \pm k_{0, \e, 0}\}$, the operator $\mathcal{J}_{\e}(k)$ is invertible, and
$$
\mathcal{J}_{\e}(k) ^{-1} \mathcal{R}_{\e}(k) = (Id - \f{\mathcal{L}_{\e}^{-1}\mathcal{K}_{\e}}{k^2})^{-1} \mathcal{L}_{\e}^{-1} \mathcal{R}_{\e}(k).
$$
Thus,
$\|\mathcal{J}_{\e}(k) ^{-1} \mathcal{R}_{\e}(k)\|_{\textbf{L}(V_{\e} \rightarrow V_{\e})}=O(\e)$ uniformly for $k\in \p W$.

By the generalized Rouche's theorem, we can conclude that for $\e$ sufficiently small, the operator $\mathcal{A}_{\e}(k)$ has the same multiplicity as the operator $\mathcal{J}_{\e}(k)$ in $W$, which is zero.
Since $\mathcal{A}_{\e}(k)$ has a pole of order two, we derive that $\mathcal{A}_{\e}(k)$ has either one characteristic value of order two or two characteristic value of order one.
This completes the proof of the lemma. \end{proof}

\medskip

We now present a systematic way to calculate the resonances. For this purpose, we need a detailed study of the operator $\mathcal{R}(k)$. Recall that
$$
R(x, y, k)
= \f{ik}{2\pi}\int_{0}^1 e^{ik|x-y|t} d t
+ k\int_{0}^1 \sin{ik|x-y|t} d t + r(x, y, k).
$$
We can decompose $R$ into three parts, $R= R_1+  R_2 +R_3$, with
\beas
R_1 &=& R(0,0,0)= \alpha_0, \\
R_2 &=& k \f{\p R}{\p k}(0, 0, 0) = k \alpha_1, \\
R_3 &=& k r_1(x, y, k)|x-y| + x\cdot r_2(x, y, k) + y\cdot r_3(x, y, k) + k^2 r_4(x, y, k),
\eeas
where $r_1$ is smooth with respect to $|x-y|$ and $r_2, r_3, r_4$ are smooth with respect to $x$ and $y$,
and all are analytic with respect to $k$.

We denote the integral operators corresponding to the kernels $R_1$, $R_2$, $R_3$ by $\mathcal{R}_1$,
$\mathcal{R}_2$ and $\mathcal{R}_3$, respectively.

\begin{lem}
The operator $\mathcal{R}$ has the decomposition
$\mathcal{R}= \mathcal{R}_1+ \mathcal{R}_2+ \mathcal{R}_3$,
where
\beas
\mathcal{R}_1 &=& |D| \alpha_0 (\cdot, \psi) \psi, \\
\mathcal{R}_2 &=& |D| \alpha_1 k (\cdot, \psi) \psi,
\eeas
and $\mathcal{R}_3$ satisfies the estimate
$$
\|\mathcal{L}_{\e}^{-1} \mathcal{R}_3\| \lesssim \e^2 + \e \cdot k^2, \quad k \in W.
$$
\end{lem}
\begin{proof}
We only need to show the estimate for $\|\mathcal{L}_{\e}^{-1} \mathcal{R}_3\|$. But this follows from an argument similar to that in the proof of Lemma \ref{lem-estimate-R}. \end{proof}

\medskip

As a consequence of the above Lemma, we see that the following identity holds for $k\in W$
$$
\mathcal{B}(k) = \mathcal{L}_{\e}^{-1} (\mathcal{R}_1 + \mathcal{R}_2) +  \mathcal{L}_{\e}^{-1}\mathcal{R}_3
= \left(\alpha_0 + k\alpha_1 \right)|D| (\cdot, \psi) \mathcal{L}_{\e}^{-1}[\psi] + O(\e^2 + \e\cdot k^2).
$$

Define
$x= (\varphi, \psi)$,
$B^{(0)} = (\mathcal{L}_{\e}^{-1}[\psi], \psi)$ and $B^{(n)} =\left( (-\mathcal{B})^n \mathcal{L}_{\e}^{-1}[\psi], \psi\right)$.

\begin{lem}  \label{lem-estimate-BB}
The following estimates hold:
$$
B^{(0)} = \f{\e c_{\Lambda}}{ |D|}, \,\,| B^{(n)}| \lesssim O(\e^{n+1}).
$$
Moreover, $B^{(1)} = B^{(1, 1)}+  B^{(1, 2)}+  B^{(1, 3)},$ where
\bea
B^{(1, 1)} &= & -\alpha_0 \left(\f{\e c_{\Lambda}}{|D|}\right)^2 \cdot |D|, \\
 B^{(1, 2)} &=& -k\alpha_1\left(\f{\e c_{\Lambda}}{|D|}\right)^2 \cdot |D|, \\
 B^{(1, 3)} &=& O(\e^3+ k^2 \e^2).
\eea
\end{lem}

\begin{proof}
It is clear that $B^{(1)}$ can be decomposed as $B^{(1)} = B^{(1, 1)}+  B^{(1, 2)}+  B^{(1, 3)} $, where
\beas
B^{(1, 1)} &=& -( \mathcal{L}_{\e}^{-1}\mathcal{R}_{1} \mathcal{L}^{-1}_{\e}[\psi], \psi), \\
B^{(1, 2)} &=& -( \mathcal{L}_{\e}^{-1}\mathcal{R}_{2} \mathcal{L}^{-1}_{\e}[\psi], \psi), \\
B^{(1, 3)} &=& -( \mathcal{L}_{\e}^{-1}\mathcal{R}_{3} \mathcal{L}^{-1}_{\e}[\psi], \psi).
\eeas
The rest of the lemma follows from a direct calculation.\end{proof}

\medskip

Recall that $k$ is the solution to the following nonlinear equation
\beas
k^2  &=&  A(k, \e)= B^{(0)}  + B^{(1, 1)} + B^{(1, 2)}+  B^{(1, 3)} + \sum_{n=2}^{\infty }B^{(n)} \\
      &=& \f{\e c_{\Lambda}}{ |D|} -\alpha_0 \left(\f{\e c_{\Lambda}}{|D|}\right)^2 \cdot |D| -
      k\alpha_1\left(\f{\e c_{\Lambda}}{|D|}\right)^2 \cdot |D| + O(\e^3)+ O(\e^2)\cdot O(k^2).
\eeas

We now solve the above nonlinear equation for $k$ in $W$.
First, by using the identity
$$
\sqrt{1+x} = 1+ \f{1}{2}x - \f{1}{4}x^2 + \ldots \quad \mbox{for}\,\,|x| < 1,
$$
we can find a matrix $F=F(k, \e)$ such that $A(k, \e)= F^2(k,\e)$. In fact,
we have
\be
F(k, \e)= \sqrt{\f{\e c_{\Lambda}}{ |D|}} \left(1- \f{1}{2}\alpha_0 \e c_{\Lambda}- \f{1}{2} \alpha_1 \e c_{\Lambda}  + O(k^2)O(\e) + O(\e^2)\right), \quad k\in W.
\ee

It is clear that the following factorization holds
$$
k^2- A(k, \e) =\big( k- F(k, \e)\big)\cdot \big( k+ F(k, \e)\big).
$$
Thus, the roots (zeros) for $k^2- A(k, \e)$ are those for $k- F(k, \e)$ and $k+ F(k, \e)$.

\begin{lem}
There exist exactly one root for each of the analytic functions $k-F(k, \e)$ and $k+F(k, \e)$ in $W$. Moreover, they have the following form
\bea
 k_{0, \e, 1} &=  k_{0, \e , 0} -\f{1}{2}\alpha_0 \sqrt{\f{c_{\Lambda}}{|D|}}c_{\Lambda} \e^{\f{3}{2}}
 -\f{1}{2} \alpha_1 \f{c_{\Lambda}^2}{|D|} \e^2 + O(\e^{\f{5}{2}}), \label{formula-k1}\\
 k_{0, \e, 2} &=  -k_{0, \e , 0} +\f{1}{2}\alpha_0 \sqrt{\f{c_{\Lambda}}{|D|}}c_{\Lambda} \e^{\f{3}{2}}
 -\f{1}{2} \alpha_1 \f{c_{\Lambda}^2}{|D|} \e^2+ O(\e^{\f{5}{2}}) . \label{formula-k2}
\eea
\end{lem}
\begin{proof}
Define
$$
F^{(0)}(k, \e)= \sqrt{\f{\e c_{\Lambda}}{ |D|}} \left(1- \f{1}{2}\alpha_0 \e c_{\Lambda}
 -\f{1}{2} \alpha_1k\e c_{\Lambda} \right).
$$
By Lemma \ref{lem-estimate-BB}, we see that
$$
F(k, \e) = F^{(0)}(k, \e) + O(k^2)O(\e) + O(\e^2), \quad k\in W.
$$

We first solve the roots for the function
$k- F^{(0)}(k, \e)$. Since the function $k- F^{(0)}(k, \e)$ is linear in $k$,
it is clear that there exists a unique root. To find it,
we use the following Ansatz
$$
k_{0, \e}= \tau_1\e^{\f{1}{2}} + \tau_2 \e+ \tau_3 \e^{\f{3}{2}}+\tau_4\e^2+ \ldots.
$$

We obtain the following equation
\[
\begin{array}{ll}
 &\tau_1\e^{\f{1}{2}} + \tau_2 \e+ \tau_3 \e^{\f{3}{2}}+\tau_4\e^2+ O(\e^{\f{5}{2}})  \\
=& \sqrt{\f{\e c_{\Lambda}}{|D|}} - \f{1}{2}\alpha_0 \sqrt{\f{c_{\Lambda}}{|D|}}c_{\Lambda}\e^{\f{3}{2}}  -
     \f{1}{2}\alpha_1 (\tau_1\e^{\f{1}{2}}+O(\e)) \sqrt{\f{ c_{\Lambda}}{|D|}} c_{\Lambda}\e^2,
        \end{array}
\]
which yields
\[
\begin{array} {ll}
\tau_1^2&= \f{ c_{\Lambda}}{|D|},  \\
\tau_2&= 0, \\
\tau_3 &= -\f{1}{2}\alpha_0\sqrt{\f{ c_{\Lambda}}{|D|}}c_{\Lambda} , \\
\tau_4 &= -\f{1}{2} \tau_1 \alpha_1 \sqrt{\f{ c_{\Lambda}}{|D|}}c_{\Lambda}=
-\f{\alpha_1 c_{\Lambda}^2}{2|D|}.
\end{array}
\]
Thus, we see that the root can be written as
$\lambda_*= \tau_1\e^{\f{1}{2}} + \tau_2 \e+ \tau_3 \e^{\f{3}{2}}+\tau_4\e^2+ O(\e^{\f{5}{2}})$.
By using Rouche's theorem, we can conclude that for sufficiently small $\e$, there exists one for the matrix-valued analytic function $k- F(k, \e)$  in the neighborhood $W_1=\{k: |k -(\tau_1\e^{\f{1}{2}} + \tau_2 \e+ \tau_3 \e^{\f{3}{2}}+\tau_4\e^2) |\leq  t \e^{\f{5}{2}}\}$ for sufficiently large $t$.  On the other hand, it is also clear that there exists only one root for $k- F(k, \e)$ in the region $W$.
Therefore, we can conclude that the unique root of $k- F(k, \e)$ in $W$ is actually in $W_1$ and it has the following form
$$
 k_{0, \e, 1} =  k_{0, \e , 0} -\f{1}{2}\alpha_0\sqrt{\f{ c_{\Lambda}}{|D|}}c_{\Lambda}\e^{\f{3}{2}}
 -\f{c_{\Lambda}^2}{2|D|}\alpha_1\e^2 + O(\e^{\f{5}{2}}).
$$
Similarly, we can show that the unique root of $k+ F(k, \e)$ in $W$ has the following form
$$
 k_{0, \e, 2} =  -k_{0, \e , 0}
 +\f{1}{2}\alpha_0\sqrt{\f{ c_{\Lambda}}{|D|}}c_{\Lambda}\e^{\f{3}{2}}
 -\f{c_{\Lambda}^2}{2|D|}\alpha_1\e^2 + O(\e^{\f{5}{2}}).
$$
This completes the proof of the lemma.\end{proof}

\medskip

As a consequence of the above lemma, we immediately obtain  the following conclusion.

\begin{lem}  \label{lem-resonance}
 There exist exactly 2M characteristic values for the analytic function $k^2-A(k, \e)$ in $W$. They have the representation
 (\ref{formula-k1}) and (\ref{formula-k2}).
\end{lem}

Finally, combining Lemmas \ref{lem-connect}, \ref{lem-characteristic1}, and \ref{lem-resonance}, we conclude that the following proposition holds.

\begin{prop}
 There exist exactly two characteristic values for the operator-value analytic function $\mathcal{A}_{\e}(k)$  in the neighborhood $W$. Moreover, they have the following asymptotics
\[
\begin{array} {ll}
 k_{0, \e, 1} &=  k_{0, \e , 0} -\f{1}{2}\alpha_0\sqrt{\f{ c_{\Lambda}}{|D|}}c_{\Lambda}\e^{\f{3}{2}}
 -\f{c_{\Lambda}^2}{2|D|}\alpha_1\e^2 + O(\e^{\f{5}{2}}),\\
 k_{0, \e, 2} &=  -k_{0, \e , 0} +\f{1}{2}\alpha_0\sqrt{\f{ c_{\Lambda}}{|D|}}c_{\Lambda}\e^{\f{3}{2}}
 -\f{c_{\Lambda}^2}{2|D|}\alpha_1\e^2 + O(\e^{\f{5}{2}}).
 \end{array}
\]
Furthermore, under the normalization condition $(\mu, \psi)=1$, the corresponding characteristic functions are gives by
\[
\begin{array} {ll}
\mu_1 &= \f{\mathcal{L}_{\e}^{-1}[\psi]}{(\mathcal{L}_{\e}^{-1}[\psi], \psi)} + O(\e), \\
\mu_2 &= \f{\mathcal{L}_{\e}^{-1}[\psi]}{(\mathcal{L}_{\e}^{-1}[\psi], \psi)} + O(\e).
 \end{array}
\]

\end{prop}

\subsection{The inhomogeneous problem }

We now are ready to solve the inhomogeneous problem
\be  \label{resonance-equation2}
 \mathcal{A}_{\e}(\mu)= (\mathcal{L}_{\e} - \f{\mathcal{K}_{\e}}{k^2} + \mathcal{R}_{\e})[\mu] =f,
\ee
where $f\in \mathcal{C}^2$ with $\|f\|_{\mathcal{C}^2} \leq O(1)$ and $k\in (\R \backslash\{0\}) \bigcap W$.
Let $\mu$ be the solution. Then,
\be \label{equ-mu1}
\mu - \f{(\mu, \psi)  \mathcal{L}_{\e, 1}^{-1}[\psi]}{k^2} = \mathcal{L}_{\e, 1}^{-1}[f].
\ee
Multiplying both sides by $\psi$, we obtain
$$
(\mu, \psi) = \f{(\mathcal{L}_{\e, 1}^{-1}[f], \psi)}{ 1- \f{A(k, \e)}{k^2}}
= (\mathcal{L}_{\e, 1}^{-1}[f], \psi)\f{k^2}{ k^2- A(k, \e)}.
$$

We need to find the inverse of $k^2-A(k, \e)$.

\begin{lem} \label{lem-inverse}
The inverse of the function $k^2-A(k, \e)$ has the following representation:
$$
\f{1}{k^2 - A(k, \e)} =  \f{1} {(k- k_{0,\e,1})(k- k_{0,\e,2})}\big(1+ O(\e)\big),
$$
where $O({\e})$ term is a smooth function in the variables $\sqrt{\e}$ and $k$ and is of the order of ${\e}$ for $k \in W$.
\end{lem}
\begin{proof}
Recall that $k^2-A(k, \e)=(k-F(k, \e))(k+F(k, \e))$. We need only consider the function $k-F(k, \e)$ and $k+F(k, \e)$.
We first investigate the function  $k-F(k, \e)$. Note that $k_{0, \e, 1}$ is the unique root of the function (with respect to the variable $k$)
$$
k- F(k, \e) =0
$$
in the disk $\{|k|\leq \f{1}{2}k_1\}$ for $\e$ small enough. Thus, the function $k- F(k, \e)$ can be written in the form
\be  \label{equ-h}
k- F(k, \e)= (k-k_{0, \e, 1} )(1+ h(k, \e))
\ee
for some analytic function $h$ in $k$. Note that the two functions $k- F(k, \e)$ and $k-k_{0, \e, 1}$  are smooth with respect to the variable $\sqrt{\e}$. We can conclude that the function $h=h(k, \e)$ is also smooth with respect to $\sqrt{\e}$.
By using the Taylor expansion, we can write $h$ in the following form
$$
h(k, \e) = h_0(k) + h_1(k)\sqrt{\e} + h_2(k, \sqrt{\e}) \e,
$$
where the functions $h_0(k)$ and $h_1(k)$ are analytic in $k$ and the function $h_2(k,t)$ is analytic in $k$ and is smooth in $t$. By comparing the coefficients of different orders of $\sqrt{\e}$ on both sides of the equation (\ref{equ-h}), we can deduce that $h_0(k)=h_1(k)=0$.
Therefore, we can conclude that $h= h_2(k, \sqrt{\e})= O(\e)$.

Similarly, we can prove that
$$
k+ F(k, \e)= (k-k_{0, \e, 2} )(1+O(\e)).
$$

Therefore, we can conclude that
\beas
\f{1}{k^2-A(k, \e)}=\f{1}{k-F(k, \e)}\f{1}{k+F(k, \e)}
&=& \f{1}{(k-k_{0, \e, 1} )(k-k_{0, \e, 2} ) } \f{1}{(1+ O(\e))(1+ O(\e))}\\
&=&\f{1} {(k- k_{0,\e,1})(k- k_{0,\e,2})}\big(1+ O({\e})\big),
\eeas
which completes the proof of the lemma.
\end{proof}

\begin{lem}  \label{lem-estimate-f}
We have $$
\mathcal{L}_{\e, 1}^{-1}[f] =
  f(0)\mathcal{L}_{\e}^{-1}[1] + O(\e^{\f{3}{2}})
 \quad \mbox{in $V_{\e}$},
$$
where the $O(\e^{\f{3}{2}})$ terms can be controlled  by $\|f\|_{\mathcal{C}^2}$.
\end{lem}
\begin{proof}

We write $f(x)= f(0)+ x g(x)$ for some smooth function $g \in \mathcal{C}^1$.
Recall that
$$
\mathcal{L}_{\e, 1}^{-1} = \sum_{n\geq 0} {\mathcal{B}}^n \mathcal{L}_{\e}^{-1},
$$
where $\|\mathcal{B}\| \leq O(\e)$.
We need only to show that
$\mathcal{L}_{\e, 1}^{-1}[xg] = O(\e^{\f{3}{2}})$ in $V_{\e}$.

It suffices to show that
$$
\|\mathcal{L}_{\e}^{-1} [xg]\|_{V_{\e}} \lesssim O(\e^{\f{3}{2}}).
$$

By the scaling identity
$$
\mathcal{L}_{\e}^{-1}[xg] = \iota_{\e,1}^{-1}\mathcal{L}_1^{-1}[\iota_{\e, 2} [{xg}]] =
\e \iota_{\e,1}^{-1}\mathcal{L}_1^{-1}[x\iota_{\e, 2} [{g}]].
$$

Thus,
$$
\|\mathcal{L}_{\e}^{-1} [xg] \|_{V_{\e}} \lesssim \e \|\iota_{\e,1}^{-1}\mathcal{L}_1^{-1} [x\iota_{\e, 2} [{g}]]\|_{V_{\e}}
= \e \e^{-\f{1}{2}}  \|\mathcal{L}_1^{-1} [x\iota_{\e, 2} [{g}]] \|_{V}
\lesssim   \e \e^{\f{1}{2}} \|\iota_{\e, 2} [{g}])\|_{V}
\lesssim \e^{\f{3}{2}}\|f\|_{\mathcal{C}^2}.
$$
The lemma is then proved.\end{proof}

\medskip

\begin{prop}  \label{prop-inhomo1}
There exists a unique solution to the integral equation (\ref{resonance-equation2}). Moreover, the solution, denoted by $\mu$, can be written as $\mu =  \mu^{(0)}+ \mu^{(1)}$,
where
\beas
\mu^{(0)}&=&  \left(\f{1}{2k_{0,\e, 0}} \big( \f{1}{k- k_{0,\e,2}} -\f{1}{k- k_{0,\e,1}}\big) \f{1}{\sqrt{|D|}} c_{\Lambda}\e + \sqrt{|D|} \right) f(0)\mathcal{L}_{\e}^{-1}[\psi], \\
\mu^{(1)}&=&  \big( \f{1}{k- k_{0,\e,2}} -\f{1}{k- k_{0,\e,1}}\big) O(\e^2)+ O(\e^{\f{3}{2}}).
\eeas
Moreover, we have
$$
(\mu , \psi)= \left(1+
\f{k_{0,\e, 0}}{2 } \big( \f{1}{k- k_{0,\e,2}} -\f{1}{k- k_{0,\e,1}}\big)\right) \f{f(0)c_{\Lambda} \e}{\sqrt{|D|}}
+  \big( \f{1}{k- k_{0,\e,2}} -\f{1}{k- k_{0,\e,1}}\big) O(\e^{\f{5}{2}}) +O(\e^2).
$$


\end{prop}
\begin{proof}

Denote by $x =(\mu , \psi)$ and $b= (\mathcal{L}_{\e, 1}^{-1}[f],  \psi )$.  By Lemma \ref{lem-estimate-f},
$$
b=  (f(0)\mathcal{L}_{\e}^{-1}[1],  \psi) + O(\e^2)= f(0) \f{1}{\sqrt{|D|}} c_{\Lambda} \e+ O(\e^2)= b^{(0)}+  O(\e^2).
$$

By Lemma \ref{lem-inverse}, we deduce that
\beas
\f{x}{k^2} = \f{1}{k^2 - A(k, \e)} b &=& \f{1+ O({\e})}{ (k- k_{0,\e,1})(k- k_{0,\e,1})}b \\
 &=& \f{1}{2k_{0,\e, 0}} \big( \f{1}{k- k_{0,\e,2}} -\f{1}{k- k_{0,\e,1}}\big) (1+ O({\e}))b.
\eeas
On the other hand,
$$
\mathcal{L}_{\e, 1}^{-1}[f] =
 f(0) \sqrt{|D|}\mathcal{L}_{\e}^{-1}[\psi] + O(\e^{\f{3}{2}})
 \quad \mbox{in $V_{\e}$},
$$
where the $O(\e^{\f{3}{2}})$ terms can be controlled  by $\|f\|_{\mathcal{C}^2}$. Similarly,
$$
 \mathcal{L}_{\e, 1}^{-1}[\psi] =   \mathcal{L}_{\e}^{-1}[\psi] + O(\e^{\f{3}{2}}) \quad \mbox{in } V_{\e}.
$$

Plugging these into the following formula
$$
\mu = \f{x}{k^2} \mathcal{L}_{\e, 1}^{-1}[\psi] + \mathcal{L}_{\e, 1}^{-1}[f],
$$
we obtain
\beas
\mu
&=& \big( \mathcal{L}_{\e}^{-1}[\psi] + O(\e^{\f{3}{2}})  \big)\f{x}{k^2}  +
\sqrt{|D|}f(0) \mathcal{L}_{\e}^{-1}[\psi] + O(\e^{\f{3}{2}}) \\
&=& \big( \mathcal{L}_{\e}^{-1}[\psi] +
 O(\e^{\f{3}{2}})\big) \f{1}{2k_{0,\e, 0}} \big( \f{1}{k- k_{0,\e,2}} -\f{1}{k- k_{0,\e,1}}\big) \big(1+ O({\e})\big) \big(b^{(0)}+O(\e^2)\big) \\
&& + \sqrt{|D|}f(0)\mathcal{L}_{\e}^{-1}[\psi] + O(\e^{\f{3}{2}}) \\
&=& \mu^{(0)}+  \mu^{(1)},
\eeas
where
$$
\mu^{(0)}=  \f{1}{2k_{0,\e, 0}} \big( \f{1}{k- k_{0,\e,2}} -\f{1}{k- k_{0,\e,1}}\big)b^{(0)}\mathcal{L}_{\e}^{-1}[\psi] + \sqrt{|D|}f(0)\mathcal{L}_{\e}^{-1}[\psi]
$$
and the remaining terms are denoted by $\mu^{(1)}$.
We can check that
$$ \mu^{(1)}= \big( \f{1}{k- k_{0,\e,2}} -\f{1}{k- k_{0,\e,1}}\big)  O(\e^{2})+ O(\e^{\f{3}{2}}).$$

The rest of the lemma follows from a straight forward calculation.
This completes the proof of the lemma. \end{proof}

%
%


\medskip

\subsection{The perturbed Green function in the exterior domain}

We are now ready to calculate the leading asymptotic of the perturbed Green function $G^{ex}_{\e}$ in the exterior domain.
For the purpose, we take the inhomogeneous term $f$ in (\ref{resonance-equation2}) to be $-G^{ex}( \cdot, x_0, k)$.
Throughout this section, we assume that the distance between $x_0$ and the opening of resonator (the origin) is much greater than $\e$, say for instance $|x_0|= O(1)$. Thus $\|f\|_{\mathcal{C}^2}$ is well-bounded and we can apply the results in the previous section.

\begin{thm} \label{thm-single}
Assume that $k\in W \bigcap \R$. Then the perturbed exterior Green function has the following asymptotic
\beas
G_{\e}^{ex}(x, x_0, k) &=& G^{ex}(x, x_0, k) \\
&& - G^{ex}(x, 0, k) G^{ex}(0, x_0, k) c_{\Lambda} \e  \\
&& -G^{ex}(x, 0, k) G^{ex}(0, x_0, k)
 \big( \f{1}{k- k_{0,\e,2}} -\f{1}{k- k_{0,\e,1}}\big)\f{(c_{\Lambda} \e)^{\f{3}{2}}}{\sqrt{|D|}}  \\
&& +   \f{O(\e^{\f{5}{2}})}{k- k_{0,\e,2}} + \f{O(\e^{\f{5}{2}})}{k- k_{0,\e,1}}+ O(\e^{2}).
\eeas
\end{thm}
\begin{proof}
By using formula (\ref{green-exterior}), we see that
$$
G_{\e}^{ex}(x, x_0, k)
= G^{ex}(x, x_0, k) + \int_{\Lambda_{\e}} G^{ex}(x, y, k) \mu(y)dy.
$$

We write  $G^{ex}(x, y, k)=  G^{ex}(x, 0, k)+ G^{ex}_1(x, y, k)$ where $G^{ex}_1(x, y, k) =y\cdot g(x, y, k)$ for some smooth function $g(x, y, k)$.
We first show the following estimate
\be   \label{estimation1}
\| G^{ex}_1(x, \cdot, k)  \|_{V^*_{\e}} = O(\e^{\f{3}{2}}).
\ee
Indeed, by Lemma \ref{scaling}, we have
$$
\| G^{ex}_1(x, \cdot, k)  \|_{V^*_{\e}} = \e^{\f{1}{2}} \| \iota_{\e, 2}[G^{ex}_1](x, \cdot, k)  \|_{V^*_{1}}
= \e^{\f{1}{2}} \e \| y \cdot g(x, \e y, k)  \|_{V^*_{1}} =  O(\e^{\f{3}{2}}).
$$
The estimate is then proved.

Now, we have
\beas
\int_{\Lambda_{\e}} G^{ex}(x, y, k) \mu(y)dy
&=& \int_{\Lambda_{\e}} G^{ex}(x, 0, k) \mu(y)dy
+ \int_{\Lambda_{\e}} G^{ex}_1(x, y, k) \mu(y)dy \\
&=& I + II.
\eeas

We analyze each of the two terms above. For the first term,
by Proposition \ref{prop-inhomo1},
\beas
I&=& \sqrt{|D|}G^{ex}(x, 0, k) (\mu, \psi)\\
&=&  - \sqrt{|D|} G^{ex}(x, 0, k) \left(1+
\f{k_{0,\e, 0}}{2 } \big( \f{1}{k- k_{0,\e,2}} -\f{1}{k- k_{0,\e,1}}\big)\right)\f{G^{ex}(0, x_0, k) c_{\Lambda} \e}{\sqrt{|D|}} \\
 && + \sqrt{|D|} G^{ex}(x, 0, k) \big( \f{1}{k- k_{0,\e,2}} -\f{1}{k- k_{0,\e,1}}\big) O(\e^{\f{5}{2}}) + O(\e^2)\\
&=&  - G^{ex}(x, 0, k) G^{ex}(0, x_0, k) c_{\Lambda} \e  \\
&& - G^{ex}(x, 0, k) G^{ex}(0, x_0, k)
 \big( \f{1}{k- k_{0,\e,2}} -\f{1}{k- k_{0,\e,1}}\big)\f{(c_{\Lambda} \e)^{\f{3}{2}}}{\sqrt{|D|}} \\
&& + \big( \f{1}{k- k_{0,\e,2}} -\f{1}{k- k_{0,\e,1}}\big) O(\e^{\f{5}{2}})+ O(\e^2).
\eeas

In order to estimate (II), note that by Proposition \ref{prop-inhomo1} and the fact that $\|\mathcal{L}_{\e}^{-1}[\psi]\|_{V_{\e}} = O(\sqrt{\e})$, we can deduce
$$
\mu= \f{O(\e)}{k- k_{0,\e,2} }+ \f{O(\e)}{k- k_{0,\e,1} } + O(\e^{\f{3}{2}}).
$$
Combining this estimate with (\ref{estimation1}), we obtain
$$
II =  \f{O(\e^{\f{5}{2}})}{k- k_{0,\e,2}} +\f{O(\e^{\f{5}{2}})}{k- k_{0,\e,1}}+ O(\e^{2}).
$$

The theorem follows immediately. \end{proof}

\section{The multiple resonators problem} \label{sec-resonator-multi}

\subsection{Preliminary}

Let the system of $M$ resonators be given as in Section \ref{sec-sub-TRE}.
We aim at deriving the resonances and the Green function for the following problem:
\bea
(\Delta + k^2)G_{\e}(x, x_0, k) = \delta(x-x_0),  \quad x \in \Om_{\e},\\  \label{helm1-1}
\f{\p G_{\e}}{\p \nu}(x, x_0, k)  = 0, \quad x \in \p \Om_{\e}, \\  \label{helm2-2}
 G_{\e} \mbox{ satisfies the radiation condition}.  \label{helm3-3-2}
\eea
Let
$$\varphi_{\e, j}=\f{\p G_{\e}}{\p \nu}(\cdot, x_0, k)|_{\Lambda_{\e, j}}.$$
Denote by $\varphi_{\e}=\f{\p G_{\e}}{\p \nu}(\cdot, x_0, k)|_{\Lambda_{\e}}
= (\varphi_{\e, 1}, \varphi_{\e, 2}, \ldots, \varphi_{\e, M} )^{t}$. We also denote by $G^{ex}(x, y, k)$ and
$G^{in}_j(x, y, k)$ the limiting problem in $\Om^{ex}$ and $D_j$, respectively.
At each aperture $\Lambda_{\e, j}$, we have the following integral equation:
\be  \label{multi-BIE}
\int_{\Lambda_{\e, j}} (G^{ex}(x, y, k) + G^{in}_j(x, y, k))\varphi_{\e, j}(y)dy
+ \sum_{l\neq j}\int_{\Lambda_{\e, j}} G^{ex}(x, y, k)\varphi_{\e, j}(y)dy + G^{ex}(x, x_0, k) =0.
\ee

We now introduce four integral operators:
\bea
\mathcal{L}_{\e, j}[f](x) &=& \int_{\Lambda_{\e, j}} \f{1}{\pi|x-y|} f(y)dy, \quad x\in \Lambda_{\e, j},\\
\mathcal{K}_{\e, j}[f](x) &=& (\int_{\Lambda_{\e, j}} \psi (y-z^{(j)}) f(y)dy) \psi(x-z^{(j)})
= \f{1}{|D|}\int_{\Lambda_{\e}} f(y)dy,  \quad x\in \Lambda_{\e, j}, \\
\mathcal{R}_{\e, j}[f](x) &=& \int_{\Lambda_{\e}} R(x-z^{(j)}, y-z^{(j)}, k) f(y)dy, \quad x\in \Lambda_{\e, j} , \\
\mathcal{R}_{\e, j, l}[f](x) &=& \int_{\Lambda_{\e, l}} G^{ex}(x, y, k) f(y)dy, \quad x\in \Lambda_{\e, j}.
\eea

Let us also define
\[ \mathcal{L}_{\e} = \left( \begin{array}{cccc}
\mathcal{L}_{\e, 1} &  &  & \\
 & \mathcal{L}_{\e, 2} &  &\\
 &  & \ldots&   \\
 &   &    &  \mathcal{L}_{\e, M}\end{array} \right),
\mathcal{K}_{\e} = \left( \begin{array}{cccc}
\mathcal{K}_{\e, 1} &  &  & \\
 & \mathcal{K}_{\e, 2} &  &\\
 &  & \ldots&   \\
 &   &    &  \mathcal{K}_{\e, M}\end{array} \right)
 \]

\[ \mathcal{R}_{\e} = \left( \begin{array}{cccc}
\mathcal{R}_{\e, 1} &  &   &  \\
  & \mathcal{R}_{\e, 2} &  & \\
 &  & \ldots&   \\
 &   &    &  \mathcal{R}_{\e, M}\end{array} \right)
+ \left( \begin{array}{cccc}
 & \mathcal{R}_{\e, 1, 2} & \mathcal{R}_{\e, 1, 3}  & \ldots \\
 \mathcal{R}_{\e, 2, 1} &  & \mathcal{R}_{\e, 2, 3}   & \ldots\\
 &  & \ldots&   \\
\mathcal{R}_{\e, M, 1} & \ldots  & \mathcal{R}_{\e, M, M-1}  &  \end{array} \right)
= \mathcal{R}_{\e, 1} + \mathcal{R}_{\e, 2}.\]

Denote by $V_{\e} = \prod_{j=1}^M V_{\e, j}$, $V_{\e}^* = \prod_{j=1}^M V_{\e, j}^*$.
For $1\leq j\leq M$, we define $\psi_j$ to be the element in $V_{\e}$ whose j-th component is the constant function $\f{1}{\sqrt{|D|}}$ and the others are zeros. With these notations, the operator $\mathcal{K}_{\e}$ can be written as
$$
\mathcal{K}_{\e}[\mu] = \sum_{j=1}^M (\mu, \psi_j) \psi_j.
$$

Denote by $f= (f_1, f_2,\ldots,f_M)^t$, where $f_j= -G^{ex}(\cdot, x_0, k)|_{\Lambda_{\e, j}}$.
We can rewrite the integral equations (\ref{multi-BIE}) in the following form
\be
 (\mathcal{L}_{\e} - \f{\mathcal{K}_{\e}}{k^2} + \mathcal{R}_{\e})[\varphi_{\e}] = f,
\ee
which is similar to the one for the single resonator case. Moreover, it is clear that the following estimate holds
$$
 \|\mathcal{L}_{\e} ^{-1} \mathcal{R}_{\e}\| \lesssim \e.
$$

We now present more detailed analysis for the operator $\mathcal{R}_{\e, 1}$ and $\mathcal{R}_{\e, 2}$.
We first consider the operator $\mathcal{R}_{\e, 1}$.

\begin{lem} \label{lem-r1}
The operator $\mathcal{R}_{\e, 1}$ has the following decomposition
\be
\mathcal{R}_{\e, 1}= \mathcal{R}_{\e, 1, 0} +\mathcal{R}_{\e, 1,1}+ \mathcal{R}_{\e, 1, 2},
\ee
where
\beas
\mathcal{R}_{\e, 1, 0}  &= & \sum_{1\leq j \leq M} \alpha_0\left( \f{\e c_{\Lambda}}{|D|} \right)^2|D| (\cdot, \psi_j)\psi_j, \\
\mathcal{R}_{\e, 1, 1}
&= & \sum_{1\leq j \leq M} k \alpha_1\left( \f{\e c_{\Lambda}}{|D|} \right)^2 |D|(\cdot, \psi_j)\psi_j, \\
\mathcal{R}_{\e, 1, 2}
&= &O(\e^3).
\eeas
\end{lem}

We next consider the operator $\mathcal{R}_{\e, 2}$. The following result holds.
\begin{lem} \label{lem-r2}
The operator $\mathcal{R}_{\e, 2}$ has the following decomposition
\be
\mathcal{R}_{\e, 2}= \mathcal{R}_{\e, 2, 0} +\mathcal{R}_{\e, 2,1},
\ee
where $\mathcal{R}_{\e, 2, 0}$ has the representation
\be
\mathcal{R}_{\e, 2, 0}[\mu] = \sum_{j=1}^M \sum_{l\neq j}G^{ex}(z^{(j)}, z^{(l)}, k) |D|(\mu, \psi_l)\psi_j
\ee
and $\mathcal{R}_{\e, 2, 1}$ satisfies the estimate
\be
\|  \mathcal{L}_{\e} ^{-1} \mathcal{R}_{\e, 2, 1} \|  \lesssim O(\e^2).
\ee
\end{lem}
\begin{proof}
Observe that for $x\in \Lambda_{\e, j}, y\in \Lambda_{\e, l}$, we have
$$
G^{ex}(x, y, k) = G^{ex}(z^{(j)}, z^{(l)}, k) + g_1(x, y, k)(x-z^{(j)}) + g_2(x, y, k)(y-z^{(l)}),
$$
where $g_1, g_2$ are smooth functions in $x, y$ and are analytic with respect to $k$.
Thus, we can decompose the operator
$\mathcal{R}_{\e, j, l}$ into two parts: $\mathcal{R}_{\e, j, l, 0}$ and $\mathcal{R}_{\e, j, l, 1}$ which correspond to the kernel $ G^{ex}(z^{(j)}, z^{(l)}, k)$ and $g_1(x, y, k)(x-z^{(j)}) + g_2(x, y, k)(y-z^{(l)})$ respectively.
By using the same method in the proof of Lemma \ref{lem-estimate-R}, the following estimate holds
$$
\| \mathcal{L}_{\e, j}^{-1} \mathcal{R}_{\e, j, l, 1} \| \lesssim  O(\e^2).
$$

Following the decomposition of the operator $\mathcal{R}_{\e, j, l}$, we define a similar decomposition for the operator $\mathcal{R}_{\e, 2}= \mathcal{R}_{\e, 2, 0}+ \mathcal{R}_{\e, 2, 1}$. It is clear that the two operators 
$\mathcal{R}_{\e, 2, 0}$ and $\mathcal{R}_{\e, 2, 1}$ have the required properties. This completes the proof of the lemma.\end{proof}

\subsection{The resonances for the multiple resonators} \label{sec-sub-resonance-multi}

We first consider the operator $\mathcal{J}_{\e} = \mathcal{L}_{\e}- \f{\mathcal{K}_{\e}}{k^2}$. Following the same approach as for the single resonator case, we can prove the lemma below.
\begin{lem}
For $\e$ small enough, there are exactly two characteristic values ( eigenvalue ) for the operator $\mathcal{J}_{\e}(k)$ in $\mathbb{C}$ and each has multiplicity $M$. Moreover, the characteristic values are given as
$$
k_{0, \e, 0} = \pm \sqrt{\f{1}{|D|} \e c_{\Lambda}}.
$$
\end{lem}
In what follows, we consider the resonances for the operator $\mathcal{A}_{\e}(k) = \mathcal{L}_{\e} - \f{\mathcal{K}_{\e}}{k^2} + \mathcal{R}_{\e}(k)$. We shall show that under the perturbation of the operator $\mathcal{R}_{\e}$, each one of the two characteristic values of $\mathcal{J}_{\e}(k)$ splits into $M$ resonances. 
We first establish the following preliminary result by using the generalized Rouche's theorem as for Lemma \ref{lem-characteristic1}.

\begin{lem} \label{lem-charact-mul}
For sufficiently small $\e$, there exist 2M characteristic values, counting multiplicity, for the operator-valued analytic function $\mathcal{A}_{\e}(k)$  in the neighborhood $W=\{k: |k|\leq \f{1}{2}k_1\}$. Moreover, they all have the asymptotic
$$
 k_{0, \e} = \pm k_{0, \e , 0} + O(\e).
$$
\end{lem}

We now present a systematic way to calculate these characteristic values (or resonances).
We assume that
$$
(\mathcal{L}_{\e} - \f{\mathcal{K}_{\e}}{k^2} + \mathcal{R}_{\e})[\varphi_{}] = 0.
$$
For $\e$ small enough, the operator $\mathcal{L}_{\e}+\mathcal{R}_{\e}$ is invertible. Moreover,
$$
(Id + \mathcal{L}_{\e}^{-1}\mathcal{R}_{\e})^{-1} \mathcal{L}_{\e}^{-1}
=\sum_{n=0}^{\infty}(-\mathcal{B}(k))^n \mathcal{L}_{\e}^{-1}
$$
where $\mathcal{B}(k)= \mathcal{L}_{\e}^{-1}\mathcal{R}_{\e}(k)$.
Thus,
$$
\varphi_{} = \f{1}{k^2} (\mathcal{L}_{\e}+\mathcal{R}_{\e})^{-1} \sum_{i=1}^M (\varphi, \psi_i) \psi_i
= \f{1}{k^2}\sum_{i=1}^M (\varphi, \psi_i) \sum_{n=0}^{\infty} (-\mathcal{B})^n \mathcal{L}^{-1}_{\e}\psi_i.
$$
Consequently,
\beas
k^2 (\varphi_{}, \psi_j) &= &
\sum_{i=1}^M (\varphi, \psi_i) \left( \sum_{n=0}^{\infty} (-\mathcal{B})^n \mathcal{L}^{-1}_{\e}[\psi_i], \psi_j\right) \\
&= &
\sum_{i=1}^M (\varphi, \psi_i) (\mathcal{L}^{-1}_{\e}[\psi_i], \psi_j) -
\sum_{i=1}^M (\varphi, \psi_i) ( \mathcal{B} \mathcal{L}^{-1}_{\e}[\psi_i], \psi_j) + \ldots.
\eeas

We define $x=(x_1,\ldots, x_M)^t$, where $x_j= (\varphi, \psi_i)$,
$B_{i,j}^{(0)} = (\mathcal{L}_{\e}^{-1}[\psi_{i}], \psi_{j} )$, and
$$B_{i,j}^{(n)} = \big((-\mathcal{B})^n \mathcal{L}_{\e}^{-1} [\psi_{i}], \psi_{j} \big).$$

Then the above equation can be rewritten equivalently as
\be
k^2 x= A(k, \e)x:=\left( B^{(0)}  +  B^{(1)} + \sum_{n=2}^{\infty }B^{(n)} \right) x.
\ee
Thus, we have actually proved the following lemma.

 \begin{lem}  \label{lem-connect-mul}
If $k$ is a characteristic value for the operator-valued analytic function $\mathcal{A}_{\e}(k)$, then
$k$ is a characteristic value for the matrix-valued function $k^2- A(k, \e)$.
 \end{lem}

\medskip

We now determine the characteristic values and the corresponding characteristic vectors for $k^2- A(k, \e)$.

Let the matrices $T=(T_{ij})_{M\times M}$ and $S=(S_{ij})_{M\times M}$ be defined as in (\ref{defTS}).
We first calculate in the following lemma the matrices $B^{(0)}$, $B^{(1)}$, and  $B^{(n)}$ for $n\geq 2$.
\begin{lem}  \label{lem-estimate-B}
We have
$$
B^{(0)} = \f{\e c_{\Lambda}}{ |D|} Id, \,\,\| B^{(n)}\| \lesssim O(\e^{n+1});
$$
Moreover, $B^{(1)} = B^{(1, 1)}+  B^{(1, 2)}+ B^{(1, 3)}$ where
\bea
B^{(1, 1)} &= & -\big(\alpha_0+ k\alpha_1 + O(\e)+ O(k^2)\big)\left( \f{\e c_{\Lambda}}{|D|} \right)^2 |D|Id, \\
 B^{(1, 2)}_{i,j} &=&-
 (\f{\e c_{\Lambda}}{ |D|} )^2|D| (1-\delta_{ij})\left( T +  kS + O(k^2)O(\e^2) \right), \\
 B^{(1, 3)} &=& O(\e^3).
\eea
\end{lem}
\begin{proof} First, it is clear that
\beas
B^{(0)}_{i,j}&=& (\mathcal{L}^{-1}_{\e}[\psi_i], \psi_j) = (\mathcal{L}^{-1}_{\e}[\psi_j], \psi_j)
= \f{1}{|D|} \e c_{\Lambda}, \\
(\mathcal{B}^n \mathcal{L}^{-1}_{\e}[\psi_i], \psi_j)| & \leq & \|\mathcal{B}\|^n \cdot
\|\mathcal{L}^{-1}_{\e} [\psi_i]\|_{V_{\e}} \cdot \|\psi_j\|_{V_{\e}^*}  \\
& \lesssim & \e^n \|\mathcal{L}^{-1}_{\e}[\psi_j]\|_{V_{\e}} \cdot \|\psi_j\|_{V_{\e}^*} \\
& = &  \e^n (\mathcal{L}^{-1}_{\e}[\psi_j], \psi_j) \\
& = & \e^n \cdot \f{1}{|D|} \e c_{\Lambda}.
\eeas

We now analyze the term $B^{(1)}_{i,j}= -( \mathcal{B} \mathcal{L}^{-1}_{\e}[\psi_i], \psi_j)$.
Recall that $\mathcal{B}= \mathcal{B}(k)= \mathcal{L}_{\e}^{-1}\mathcal{R}_{\e}(k)$ and
$\mathcal{R}_{\e}= \mathcal{R}_{\e, 1} + \mathcal{R}_{\e, 2} = \mathcal{R}_{\e, 1}+ \mathcal{R}_{\e, 2, 0}+ \mathcal{R}_{\e, 2, 1}$. We can decompose $B^{(1)}_{i,j}$ into the following three terms
\beas
B^{(1,1)}_{i,j}&=& ( -\mathcal{L}_{\e}^{-1}\mathcal{R}_{\e,1} \mathcal{L}^{-1}_{\e}[\psi_i], \psi_j), \\
B^{(1,2)}_{i,j}&=&( -\mathcal{L}_{\e}^{-1}\mathcal{R}_{\e, 2,0} \mathcal{L}^{-1}_{\e}[\psi_i], \psi_j), \\
B^{(1,3)}_{i,j}&=&( -\mathcal{L}_{\e}^{-1}\mathcal{R}_{\e, 2,1} \mathcal{L}^{-1}_{\e}[\psi_i], \psi_j).
\eeas

Then the rest of the proof follows from Lemmas \ref{lem-r1} and \ref{lem-r2} and the following identities:
$$
G^{ex}(z^{(i)},z^{(j)}, 0)= \f{1}{2 \pi |z^{(i)}-z^{(j)}|}, \quad
\f{\p G^{ex}(z^{(i)},z^{(j)}, 0)}{\p k} = \f{\sqrt{-1}}{2 \pi}.
$$
\end{proof}
\medskip

We are ready to find the characteristic values and the corresponding characteristic vectors for $k^2- A(k, \e)$.

Observe that
\beas
A(k, \e)&=& \f{\e c_{\Lambda}}{ |D|} Id- \big(\alpha_0+ T \big)\left( \f{\e c_{\Lambda}}{|D|} \right)^2|D|
- (\f{\e c_{\Lambda}}{ |D|} )^2|D|kS + O(k^2)O(\e^2) + O(\e^3)\\
&=& \f{\e c_{\Lambda}}{ |D|}\left(Id - \big(\alpha_0+ T \big) \e c_{\Lambda}
- kS \e c_{\Lambda} + O(k^2)O(\e) + O(\e^2)    \right) \\
&=& \f{\e c_{\Lambda}}{ |D|}(Id + O(\e)).
\eeas
By using the identity
$$
\sqrt{1+x} = 1+ \f{1}{2}x - \f{1}{4}x^2 + \ldots \quad \mbox{for}\,\,|x| < 1,
$$
we can find a matrix $F=F(k, \e)$ such that $A(k, \e)= F^2(k,\e)$. In fact,
we have
\be
F(k, \e)= \sqrt{\f{\e c_{\Lambda}}{ |D|}} \left(Id- \f{1}{2}\big(\alpha_0+ T \big) \e c_{\Lambda}
  - \f{1}{2}k S\e c_{\Lambda} + O(k^2)O(\e) + O(\e^2)\right).
\ee

It is clear that the following factorization holds
$$
k^2- A(k, \e) =\big( k- F(k, \e)\big)\cdot \big( k+ F(k, \e)\big).
$$
Thus, the characteristic values and the corresponding characteristic vectors for $k^2- A(k, \e)$ are those for $k- F(k, \e)$ and $k+ F(k, \e)$, which we investigate henceforth.

\begin{lem}
There exist exactly M characteristic values for the matrix-valued analytic function $k-F(k, \e)$ and $k+F(k, \e)$ in the neighborhood $W$, respectively. More precisely, for $1\leq j \leq M$, the characteristics values are given as
(\ref{formula-k1-1}) and  (\ref{formula-k1-2}).
Moreover, after normalization, the corresponding characteristic vectors are given as
\bea
Y_{j,1} &= Y_j + \sqrt{\e}\sum_{i\neq j} \f{1}{\beta_j- \beta_i}\sqrt{\f{c_{\Lambda}}{|D|}}Y_i Y_i^t S Y_j + O(\e), \label{formula-Y1} \\
Y_{j,2} &= Y_j - \sqrt{\e}\sum_{i\neq j} \f{1}{\beta_j- \beta_i}\sqrt{\f{c_{\Lambda}}{|D|}}Y_i Y_i^t S Y_j + O(\e).\label{formula-Y2}
\eea
\end{lem}
\begin{proof}
Step 1. Define
$$
F^{(0)}(k, \e)= \sqrt{\f{\e c_{\Lambda}}{ |D|}} \left(Id- \f{1}{2}\big(\alpha_0+ T \big) \e c_{\Lambda} \right).
$$
By Lemma \ref{lem-estimate-B},
$$
F(k, \e) = F^{(0)}(k, \e) + O(k)O(\e)+ O(k^2)O(\e) + O(\e^2),  \,\,\, k\in W.
$$

We first find the characteristic values and the corresponding characteristic vectors for the matrix-valued function
$k- F^{(0)}(k, \e)$.

It is clear that
\[
Y^{-1}\big(k- F^{(0)}(k, \e)\big)Y = Q(k, \e)=: \left( \begin{array}{cccc}
Q_1(k, \e) &  &  & \\
 & Q_2(k, \e) &  &\\
 &  &\ldots&   \\
 &   &    &  Q_M(k, \e)\end{array} \right),
 \]
where $Q_{jj}(k, \e)= Q_j(k, \e) = k - \sqrt{\f{\e c_{\Lambda}}{ |D|}} + \f{1}{2}(\alpha_0+ \beta_j )\left(\f{ c_{\Lambda}}{|D|}\right)^{\f{1}{2}}c_{\Lambda}\e^{\f{3}{2}}$.
Thus, we obtain $M$ linear equations
$$
Q_j(k, \e)=0, \quad 1\leq j \leq M,
$$
whose solutions are given by
$$
\lambda_{j, \e}= \sqrt{\f{ c_{\Lambda}}{|D|}} \e^{\f{1}{2}} -\f{1}{2}(\alpha_0+ \beta_j )\left(\f{ c_{\Lambda}}{|D|}\right)^{\f{1}{2}}c_{\Lambda}\e^{\f{3}{2}},  \quad 1\leq j \leq M.
$$
%

Therefore, we can conclude that each of the solutions $\lambda_{j, \e}$ gives a characteristic value to the matrix-valued analytic function $k- F^{(0)}(k, \e)$ and the corresponding characteristic vector is $Y_j$.

\medskip

Step 2.
We now apply the generalized Rouche's theorem to obtain the existence of characteristic values for $k-F(k, \e)$.
Observe that
\[
\big(k- F^{(0)}(k, \e)\big)^{-1}= Y\left( \begin{array}{cccc}
Q_1(k, \e)^{-1} &  &  & \\
 & Q_2(k, \e)^{-1} &  &\\
 &  & \ldots&   \\
 &   &    &  Q_M(k, \e)^{-1}\end{array} \right)Y^t,
\]

For each $1\leq j\leq M$,
we define domain $W_j=: \{k: |k- \lambda_{j, \e}|\leq t_j \e^{2}\}$ where $t_j>0$. Since
$F(k, \e)-F^{(0)}(k, \e)= O(k)O(\e^{\f{3}{2}}) + O(k^2)O(\e^{\f{3}{2}}) + O(\e^{\f{5}{2}})$ and the $\beta_j$'s are pairwise different, we can deduce that for $\e$ sufficiently small, there exists $t_j$ such that the following inequality holds
$$
\| (F^{(0)}(k, \e))^{-1} \big( F(k, \e)-F^{(0)}(k, \e) \big)\| <1  \quad \mbox{for  }\, k \in \p W_j.
$$

Then the generalized Rouche's theorem yields that there exists exactly one characteristic value $k_{0,\e,j, 1}$ for $k-F(k, \e)$ in the domain $W_j$.
Therefore, we can conclude that for sufficiently small $\e$, there exist M characteristic values, for the matrix-valued analytic function $k- F(k, \e)$  in the domain $\bigcup_{j=1}^{M} W_j$. On the other hand, it is easy to show that there are M characteristic values for $k- F(k, \e)$ in the domain $W= \{k: |k| \leq \f{1}{2}k_1\}$. Therefore, the characteristic values we just calculated are exactly the characteristic values for $k- F(k, \e)$ in the $W$.

\medskip

Step 3.
We determine the forms of the characteristic vectors.
We assume that
$$
\big(k_{0, \e, j,1} - F(k_{0,\e,j,1}, \e)\big)Y_{j,\e}= 0, \quad \| Y_{j,\e}\| =O(1).
$$

We first show that $Y_{j,\e} = Y_j + O(\sqrt{\e})$.
Indeed, note that
$$
k_{0, \e, j, 1} - F(k_{0,\e, j, 1}, \e) = k_{0, \e, j, 1} - F^{(0)}(k_{0,\e,j, 1}, \e)+ O(\e^{2}),
$$
and
\[
k_{0, \e, j, 1} - F^{(0)}(k_{0,\e,j, 1}, \e) =  Y Q(k_{0,\e,j, 1}, \e)Y^t,
\]

We get
\[
Q(k_{0,\e,j, 1}, \e) Y ^t Y_{j,\e} = O(\e^{2}).
\]
Observe that
\beas
Q_i(k_{0, \e, j, 1}, \e) =  O(\e^{2}) \quad \mbox{for } \,i=j, \\
Q_i(k_{0, \e, j, 1}, \e) \geq  O(\e^{\f{3}{2}}) \quad \mbox{for } \,i \neq j,
\eeas
where we used the assumption that $\beta_j$'s are pairwise different in the second inequality above.
Therefore, $Y_j \cdot Y_{j, \e} = O(1)$ and $Y_i \cdot Y_{j, \e} =  O(\sqrt{\e})$ for $i \neq j$. It follows that $Y_{j, \e}$ can be written in the form
$$
Y_{j, \e} = Y_j +  O(\sqrt{\e}).
$$

\medskip

Step 4. From the results in the previous two steps, we can write the $j$-th characteristic value of $k-F(k,\e)$ and its associated
characteristic vector in the following form
\beas
k_{0,\e, j, 1} &=& \lambda_{j,\e} + \tau_{4,j}\e^2 + O(\e^{\f{5}{2}}), \\
Y_{j, \e} &=& Y_j + \e^{\f{1}{2}}Y^{(1)}_j + O(\e^{}).
\eeas
We now determine $\tau_{4,j}$ and $Y^{(1)}_j$.
It is important to note that we need to impose the following normalization condition
\be
 (Y_j,  Y^{(1)}_j) =0
\ee
in order to uniquely determine the unknown constants.
We write
$$
 Y^{(1)}_j = \sum_{i\neq j} y_iY_i.
$$

We have
$$
\big(k_{0, \e, j, 1}-F(k_{0, \e, j, 1},\e)\big)Y_{j, \e}=0.
$$

Note that
\beas
k_{0, \e, j, 1}- F(k_{0, \e, j, 1}, \e) &=&
 Y \left( Q(k_{0, \e, j, 1}, \e) +\f{1}{2}\sqrt{\f{c_{\Lambda}}{|D|}}\cdot c_{\Lambda} \e^{\f{3}{2}}
k_{0, \e, j, 1}Y^{t}SY + O(\e^{\f{5}{2}})  \right)Y^t\\
&=&
 Y \left( Q(k_{0, \e, j, 1}, \e) +\f{1}{2}\f{c_{\Lambda}^2}{|D|} \e^{2}
Y^{t}SY + O(\e^{\f{5}{2}})  \right)Y^t.
\eeas
Thus,
\[
  \left(  Q(k_{0, \e, j, 1}, \e)
+\f{1}{2}\f{c_{\Lambda}^2}{|D|}\e^{2}
Y^{t}SY  \right) Y^t \left(Y_j + \e^{\f{1}{2}}Y^{(1)}_j + O(\e)\right)= O(\e^{\f{5}{2}}).
\]
Since $Q(k_{0, \e, j, 1}, \e)= O(\e^{\f{3}{2}})$, we further get
\[
 Q(k_{0, \e, j, 1}, \e) \left(Y^t Y_j + \e^{\f{1}{2}}Y^t Y^{(1)}_j\right) =- \f{1}{2}\f{c_{\Lambda}^2}{|D|}\e^{2}
Y^{t}SY_j + O(\e^{\f{5}{2}}).
\]

It is easy to see that $Y^t Y_j= e_j$ and $Y^t Y^{(1)}_j= \sum_{i\neq j} y_ie_i$. Apply both sides of the above equation from left by $e_i^t$, $i=1,2,\ldots,M$,  we can derive that
\beas
Q_i(k_{0, \e, j, 1}) y_i&= & -\f{1}{2}\f{c_{\Lambda}^2}{|D|}\e^2 Y_i^t SY_j, \quad i\neq j, \\
Q_j(k_{0, \e, j, 1}) &=& -\f{1}{2}\f{c_{\Lambda}^2}{|D|}\e^2 Y_j^t SY_j.
\eeas
Since
\[
Q_i(k_{0, \e, j, 1}) = \f{1}{2}\sqrt{\f{c_{\Lambda}}{|D|}}c_{\Lambda} (\beta_i- \beta_j)\e^{\f{3}{2}} + O(\e^2), \quad i\neq j, \quad Q_j(k_{0, \e, j, 1}) = \tau_{4, j} \e^2,
\]
we obtain
$$
\tau_{4, j}= - \f{1}{2}\f{c_{\Lambda}^2}{|D|} Y_j^t SY_j,
$$
and
$$ y_i = \f{1}{\beta_j- \beta_i}\sqrt{\f{c_{\Lambda}}{|D|}} Y_i^t SY_j,
$$
which proves (\ref{formula-k1-1}) and (\ref{formula-Y1}) immediately.
By a similar procedure, we can prove (\ref{formula-k1-2}) and (\ref{formula-Y2}). This completes the proof of the lemma. \end{proof}

\medskip

As a consequence of the above result, the following lemma holds.
\begin{lem}  \label{lem-charact-a-mul}
There exist exactly 2M characteristic values of order one for the matrix-valued analytic function $k^2-A(k, \e)$ in the neighborhood $W$. They have the forms as given in (\ref{formula-k1-1}) and (\ref{formula-k1-2}).
Moreover, the corresponding characteristic vectors are given by (\ref{formula-Y1}) and (\ref{formula-Y2}).
\end{lem}

Finally, Proposition \ref{prop-resonances} follows from Lemma \ref{lem-charact-a-mul} and \ref{lem-connect-mul}.

\subsection{The inhomogeneous problem }

We now solve the inhomogeneous problem
\be  \label{resonance-equation-multi}
 \mathcal{A}_{\e}[\mu]= (\mathcal{L}_{\e} - \f{\mathcal{K}_{\e}}{k^2} + \mathcal{R}_{\e})[\mu] =f.
\ee
where $f\in \mathcal{C}^{\infty}$ for $k\in W\backslash \{0\}$ and $k\in \R$.
Let $\mu$ be the solution, then
\be \label{equ-mu-multi}
\mu - \sum_{1\leq j\leq M}\f{(\mu, \psi_j)  \mathcal{L}_{\e, 1}^{-1} [\psi_j]}{k^2} = \mathcal{L}_{\e, 1}^{-1}[f].
\ee
Denote by
\beas
\psi &=& (\psi_1, \psi_2,\ldots,\psi_M), \\
\mathcal{L}_{\e, 1}^{-1}[\psi] &=& (\mathcal{L}_{\e, 1}^{-1} \psi_1,  \mathcal{L}_{\e, 1}^{-1} [\psi_2],\ldots, \mathcal{L}_{\e, 1}^{-1}[\psi_M])^t,\\
x&=&(x_1, x_2,\ldots,x_M)^t =\big((\mu, \psi_1),(\mu, \psi_2),\ldots, (\mu, \psi_M) \big)^t, \\
b&=&(b_1, b_2,\ldots,b_M)^t= \big(( \mathcal{L}_{\e, 1}^{-1}[f], \psi_1), ( \mathcal{L}_{\e, 1}^{-1}[f], \psi_2),\ldots,( \mathcal{L}_{\e, 1}^{-1}[f], \psi_M)\big)^t\\
\mathcal{F}&=& (f(z^{(1)}), f(z^{(2)}),\ldots, f(z^{(M)}))^t.
\eeas
Recall that $ A(k, \e)_{i,j} = (\mathcal{L}_{\e, 1}^{-1} [\psi_j], \psi_i )$. Apply $k^2 \psi_i$ on both sides of (\ref{equ-mu1}) to obtain
\be  \label{equation-x-mul}
\big(k^2-A(k, \e)\big)x =k^2b.
\ee
We now derive inverse for the matrix $k^2-A(k, \e)$.

\begin{lem}  \label{lem-inverse-mul}
The inverse of the matrix $k^2-A(k, \e)$ has the following representation:
$$
\big(k^2 - A(k, \e)\big)^{-1} =  Y D_2^{-1} D_1^{-1} Y^{t} + Q(k, \e),
$$
where $D_1$ and $D_2$ are the diagonal matrices with $(D_1)_{jj}= k- k_{0,\e, j, 1}$ and $(D_2)_{jj}= k- k_{0, \e, j,2}$, and
$Q(k, \e)$ is a matrix with
$$
Q_{ij}(k, \e) = \f{O(1)}{ (k- k_{0,\e,j, 1})(k- k_{0, \e, j,2})} = \f{O(1)}{k- k_{0,\e,j, 1}} +\f{O(1)}{k- k_{0,\e, j, 2}}.
$$
\end{lem}
\begin{proof}
Recall that $k^2-A(k, \e)=\big(k-F(k, \e)\big)\big(k+F(k, \e)\big)$. We first find the inverse for the matrices $k-F(k, \e)$ and $k+F(k, \e)$. Note that for each $1\leq j\leq M$,
$$
\big(k_{0, \e, j, 1} -F(k_{0,\e, j,1}, \e)\big) Y_{j,\e} = 0
$$
By a similar argument as in the proof of Lemma \ref{lem-inverse}, we can derive that
$$
\big(k -F(k, \e)\big) Y_{j,\e} = (k- k_{0, \e,j, 1}) \big( Y_j+ O(\sqrt{\e})\big).
$$
Thus we have
$$
(k -F(k, \e))(Y+ O(\sqrt{\e}))  =  (Y+ O(\sqrt{\e})) D_1.
$$

This yields
$$
(k -F(k, \e)) =  (Y+ O(\sqrt{\e})) D_1 (Y+ O(\sqrt{\e}))^{-1},
$$
and consequently,
$$
(k -F(k, \e))^{-1} = (Y+ O(\sqrt{\e})) D_1^{-1} (Y+ O(\sqrt{\e}))^{-1}.
$$

Similarly, we have
$$
(k +F(k, \e))^{-1} = (Y+ O(\sqrt{\e})) D_2^{-1} (Y+ O(\sqrt{\e}))^{-1}.
$$

It follows that
$$
\big(k^2 - A(k, \e)\big)^{-1} = \big(Y+ O(\sqrt{\e})\big) D_2^{-1} \big((Y+ O(\sqrt{\e})\big)^{-1} \big(Y+ O(\sqrt{\e})\big) D_2^{-1} (Y+ O(\sqrt{\e}))^{-1}.
$$
Since
\beas
\big((Y+ O(\sqrt{\e})\big)^{-1} \big(Y+ O(\sqrt{\e})\big) &= Id + O(\sqrt{\e}),  \\
(Y+ O(\sqrt{\e}))^{-1}= Y^{-1} + O(\sqrt{\e})&= Y^{t} + O(\sqrt{\e}),
\eeas
we can deduce the conclusion of the lemma by a
straightforward calculation. This completes the proof of the lemma.\end{proof}

\medskip
We are ready to establish the following proposition.
\begin{prop}  \label{prop-inhomo1-mul}
There exists a unique solution to the integral equation (\ref{resonance-equation-multi}). Moreover, the solution, denoted by $\mu$, has the following form
$$
\mu =: \mu^{(0)}+ \mu^{(1)},
$$
where
\beas
\mu^{(0)}&=&
\mathcal{L}_{\e}^{-1}[\psi] \left[\sum_{j=1}^M \f{1}{2k_{0,\e,0}} \left( \f{1}{k-k_{0,\e,j,2}} -\f{1}{k-k_{0,\e,j,1}}\right)
Y_j Y_j^t \f{c_{\Lambda}\e}{\sqrt{|D|}} + Id \right] \mathcal{F},\\
\mu^{(1)}&=&
\sum_{j=1}^M  \f{1}{k-k_{0,\e,j,1}} O(\e^{\f{3}{2}}) + \sum_{j=1}^M  \f{1}{k-k_{0,\e,j,2}} O(\e^{\f{3}{2}})
+ O(\e^{\f{3}{2}}).
\eeas
Moreover,
\beas
x= \big((\mu, \psi_1),(\mu, \psi_2),\ldots,(\mu, \psi_M) \big)^t
 &=&  \sum_{j=1}^M \left( \f{1}{k-k_{0,\e,j,2}} -\f{1}{k-k_{0,\e,j,1}}\right) Y_j Y_j^t   \f{(c_{\Lambda}\e)^{\f{3}{2}}}{|D|}\mathcal{F}+ \f{c_{\Lambda}\e}{|D|}\mathcal{F} \\
  && + \sum_{j=1}^M  \f{1}{k-k_{0,\e,j,1}} O(\e^{2}) + \sum_{j=1}^M  \f{1}{k-k_{0,\e,j,2}} O(\e^{2})
+ O(\e^{2}).
\eeas
\end{prop}
\begin{proof}
First, by Lemmas \ref{equation-x-mul} and \ref{lem-inverse-mul}, we have
$$
x= \left(YD_1^{-1}D_2^{-1}Y^t + Q(k, \e) \right) k^2 b.
$$

Next, note that
$$
\mathcal{L}_{\e, 1}^{-1}[f] =
 \sum_{1\leq j\leq M}
 f(z^{(j)}) \sqrt{|D|}\mathcal{L}_{\e}^{-1}[\psi_j] + O(\e^{\f{3}{2}})= \sqrt{|D|}(\mathcal{L}_{\e}^{-1}[\psi])\mathcal{F} + O(\e^{\f{3}{2}})
 \quad \mbox{in $V_{\e}$},
$$
where the $O(\e^{\f{3}{2}})$ terms can be controlled  by $\|f\|_{\mathcal{C}^2}$. Thus
$$
b=\big(( \mathcal{L}_{\e, 1}^{-1}[f], \psi_1), ( \mathcal{L}_{\e, 1}^{-1}[f], \psi_2),\ldots,( \mathcal{L}_{\e, 1}^{-1}[f], \psi_M)\big)^t = b^{(0)}+ O(\e^2)= \f{c_{\Lambda}\e}{\sqrt{|D|}}\mathcal{F}+ O(\e^2).
$$

On the other hand, it is clear that
$$
 \mathcal{L}_{\e, 1}^{-1} [\psi_j] =   \mathcal{L}_{\e}^{-1}[\psi_j] + O(\e^{\f{3}{2}}) \quad \mbox{in } V_{\e}.
$$

It follows that
\beas
\mu
&=& \sum_{1\leq j\leq M}\f{x_j \mathcal{L}_{\e, 1}^{-1}[\psi_j]}{k^2} + \mathcal{L}_{\e, 1}^{-1} [f]\\
&=& ( \mathcal{L}_{\e, 1}^{-1}[\psi] )\f{1}{k^2} x +
\sqrt{|D|}(\mathcal{L}_{\e}^{-1}[\psi])\mathcal{F} + O(\e^{\f{3}{2}}) \\
&=& \big( \mathcal{L}_{\e}^{-1}[\psi] + O(\e^{\f{3}{2}})\big) \left(YD_1^{-1}D_2^{-1}Y^t + Q(k, \e) \right) (b^{(0)}+ O(\e^2)) \\
&& + \sqrt{|D|}(\mathcal{L}_{\e}^{-1}[\psi])\mathcal{F} + O(\e^{\f{3}{2}}) \\
&=: & I + II,
\eeas
where
\beas
I &=&  \mathcal{L}_{\e}^{-1}[\psi] \left( YD_1^{-1}D_2^{-1}Y^t b^{(0)} + \sqrt{|D|}\mathcal{F} \right) \\
II &=&
 O(\e^{\f{3}{2}}) \left(YD_1^{-1}D_2^{-1}Y^t + Q(k, \e) \right) (b^{(0)}+ O(\e^2))\\
&& +  \mathcal{L}_{\e}^{-1}[\psi]  \left( Q(k, \e)  \big(b^{(0)}+ O(\e^2)\big) +  YD_1^{-1}D_2^{-1}Y^t O(\e^2)\right) + O(\e^{\f{3}{2}}).
\eeas

Note that $$b^{(0)}= \f{c_{\Lambda}\e}{\sqrt{|D|}}\mathcal{F}= O(\e), \quad \mathcal{L}_{\e}^{-1} [\psi]= O(\sqrt{\e}),$$
$$D_1^{-1}D_{2}^{-1} = \f{1}{k_{0, \e, 0}}(D_2^{-1}- D_1^{-1})\big(1+ O(\e)\big),$$ and
$$Q_{ij}(k, \e) = \f{O(1)}{k- k_{0,\e,j,1}} +\f{O(1)}{k- k_{0,\e,j,2}}.$$

By straightforward calculation, we can obtain
\beas
I &=&  \mathcal{L}_{\e}^{-1}[\psi] \left[\sum_{j=1}^M \f{1}{2k_{0,\e,0}} \left( \f{1}{k-k_{0,\e,j,2}} -\f{1}{k-k_{0,\e,j,1}}\right)
  Y_j Y_j^t \f{c_{\Lambda}\e}{\sqrt{|D|}} + \sqrt{|D|}Id \right]\mathcal{F} \\
&& +\sum_{j=1}^M  \f{1}{k-k_{0,\e,j,1}} O(\e^{2}) + \sum_{j=1}^M  \f{1}{k-k_{0,\e,j,2}} O(\e^{2}),\\
II &=&
\sum_{j=1}^M  \f{1}{k-k_{0,\e,j,1}} O(\e^{2}) + \sum_{j=1}^M  \f{1}{k-k_{0,\e,j,2}} O(\e^{2})
+ O(\e^{\f{3}{2}})
\eeas

Now, we define
$$
\mu^{(0)}= \mathcal{L}_{\e}^{-1}[\psi] \left[\sum_{j=1}^M \f{1}{2k_{0,\e,0}} \left( \f{1}{k-k_{0,\e,j,2}} -\f{1}{k-k_{0,\e,j,1}}\right)
  Y_j Y_j^t \f{c_{\Lambda}\e}{\sqrt{|D|}} + \sqrt{|D|} \right]\mathcal{F},
$$
and $\mu^{(1)}= \mu -\mu^{(0)}$. This yields the desired decomposition and estimate for $\mu$.

As a consequence, we can also deduce that
$$
x=\big((\mu, \psi_1),(\mu, \psi_2),\ldots,(\mu, \psi_M) \big)^t= x^{(0)}+ x^{(1)},
$$
where
\beas
 x^{(0)}
 &=&  \sum_{j=1}^M \f{1}{2k_{0,\e,0}} \left( \f{1}{k-k_{0,\e,j,2}} -\f{1}{k-k_{0,\e,j,1}}\right) Y_j Y_j^t \sqrt{|D|} \big(\f{c_{\Lambda}\e}{|D|}\big)^2 \mathcal{F} + \f{c_{\Lambda}\e}{\sqrt{|D|}}\mathcal{F},\\
 &=&  \sum_{j=1}^M \left( \f{1}{k-k_{0,\e,j,2}} -\f{1}{k-k_{0,\e,j,1}}\right) Y_j Y_j^t \f{(c_{\Lambda}\e)^{\f{3}{2}}}{|D|}\mathcal{F} + \f{c_{\Lambda}\e}{\sqrt{|D|}}\mathcal{F},\\
  x^{(1)} &=& \sum_{j=1}^M  \f{1}{k-k_{0,\e,j,1}} O(\e^{2}) + \sum_{j=1}^M  \f{1}{k-k_{0,\e,j,2}} O(\e^{2})
+ O(\e^{2}).
\eeas
This completes the proof of the proposition.\end{proof}


\subsection{Proof of Theorem \ref{thm-multi}} \label{sec-sub-green-multi}


By taking the inhomogeneous term $f$ in (\ref{resonance-equation-multi}) to be $-G^{ex}( \cdot, x_0, k)$, we can apply Proposition \ref{prop-inhomo1-mul} to obtain the leading asymptotic of the perturbed Green function $G^{ex}_{\e}$ in the exterior domain. We now give more details below.

First, by formula (\ref{green-exterior}), we see that
$$
G_{\e}^{ex}(x, x_0, k)
= G^{ex}(x, x_0, k) + \int_{\Lambda_{\e}} G^{ex}(x, y, k) \mu(y)dy.
$$

For each $1\leq j \leq M$, we have
$$G^{ex}(x, y, k)=  G^{ex}(x, z^{(j)}, k)+G^{ex}_j(x, y, k),$$
where $G^{ex}_j(x, y, k) =(y-z^{(j)}) \cdot g_j(x, y, k)$
for some smooth function $g_j(x, y, k)$.
As in the proof of Theorem \ref{thm-single}, we have
\be  \label{estimation1-mul}
\|G^{ex}_j(x, \cdot, k)\|_{V_{\e, j}^*} \leq O(\e^{\f{3}{2}} ).
\ee
We decompose the integral $\ds \int_{\Lambda_{\e}} G^{ex}(x, y, k) \mu(y)dy$ in the following way:
\beas
\int_{\Lambda_{\e}} G^{ex}(x, y, k) \mu(y)dy
&=& \sum_{j=1}^M \int_{\Lambda_{\e, j}} G^{ex}(x, y, k) \mu_j(y)dy \\
&=& \sum_{j=1}^M \int_{\Lambda_{\e, j}} G^{ex}(x, z^{(j)}, k) \mu_j^{(0)}(y)dy +
\sum_{j=1}^M \int_{\Lambda_{\e, j}} G^{ex}(x, z^{(j)}, k) \mu_j^{(1)}(y)dy \\
&& + \sum_{j=1}^M \int_{\Lambda_{\e, j}} G^{ex}_j(x, y, k) \mu_j(y)dy \\
&=& I + II+ III.
\eeas

We next investigate each of the above mentioned terms.
It is clear that
\beas
I &=&  \sqrt{|D|} \mathcal{G}(x, k)^t \big((\mu, \psi_1),(\mu, \psi_2),\ldots, (\mu, \psi_M) \big)^t \\
 &=& - \sqrt{|D|} \mathcal{G}(x, k)^t \sum_{j=1}^M \left( \f{1}{k-k_{0,\e,j,2}} -\f{1}{k-k_{0,\e,j,1}}\right) Y_j Y_j^t  \f{(c_{\Lambda}\e)^{\f{3}{2}}}{|D|}\mathcal{G}(x_0, k) \\
 && - \sqrt{|D|} \mathcal{G}(x, k)^t \f{c_{\Lambda}\e}{\sqrt{|D|}}\mathcal{G}(x_0, k) \\
&=&  - \sum_{j=1}^M \left( \f{1}{k-k_{0,\e,j,2}} -\f{1}{k-k_{0,\e,j,1}}\right)\f{(c_{\Lambda}\e)^{\f{3}{2}}}{\sqrt{|D|}} \mathcal{G}(x, k)^t Y_j Y_j^t \mathcal{G}(x_0, k) \\
 && -  \mathcal{G}(x, k)^t \mathcal{G}(x_0, k) c_{\Lambda}\e.
\eeas

To estimate (II), note that by Proposition \ref{prop-inhomo1-mul}
$$
\mu^{(1)}= \sum_{1\leq j \leq M } \left(\f{O(\e^{\f{3}{2}})}{k- k_{0,\e,j, 2} }+ \f{O(\e^{\f{3}{2}})}{k- k_{0,\e,j, 1} }\right) + O(\e^{\f{3}{2}}).
$$
Combining this estimate with the fact that $\|1 \|_{V_{\e,j}^*} =O(\sqrt{\e})$,
we obtain
$$
II =  \sum_{1\leq j \leq M } \left(\f{O(\e^{2})}{k- k_{0,\e,j, 2} }+ \f{O(\e^{2})}{k- k_{0,\e,j, 1} }\right) + O(\e^{2}).
$$

Finally, we estimate (III). Again by Proposition \ref{prop-inhomo1-mul}, we have
$$
\mu= \sum_{1\leq j \leq M } \left(\f{O(\e)}{k- k_{0,\e,j, 2} }+ \f{O(\e)}{k- k_{0,\e,j, 1} }\right) + O(\e^{\f{3}{2}}).
$$
This together with
(\ref{estimation1-mul}) yields
$$
III =   \sum_{1\leq j \leq M } \left(\f{O(\e^{\f{5}{2}})}{k- k_{0,\e,j, 2} }+ \f{O(\e^{\f{5}{2}})}{k- k_{0,\e,j, 1} }\right) + O(\e^{2})
$$

The theorem follows immediately.

\section{Proof of Super-resolution by a system of sub-wavelength resonators} \label{sec-SR-proof}

\subsection{Proof of Theorem \ref{thm-super-resolution-1}} \label{sec-sub-proof-super-1}

We analyze the function $\Im{{G}_{\e}^{ex}(x, x_0, k)}$ for a fixed frequency $k\geq 0$.
Recall that for $k\in W$,
\beas
G_{\e}^{ex}(x, x_0, k) &=& G^{ex}(x, x_0, k) - {\e}c_{\Lambda}\sum_{1\leq j \leq M}G^{ex}(z^{(j)}, x_0, k)G^{ex}(x, z^{(j)}, k)\\
&& - \sum_{j=1}^M \left( \f{1}{k-k_{0,\e,j,2}} -\f{1}{k-k_{0,\e,j,1}}\right)\f{(c_{\Lambda}\e)^{\f{3}{2}}}{\sqrt{|D|}} \mathcal{G}(x, k)^t Y_j Y_j^t \mathcal{G}(x_0, k) \\
&& + \sum_{1\leq j \leq M } \left(\f{O(\e^{2})}{k- k_{0,\e,j, 2} }+ \f{O(\e^{2})}{k- k_{0,\e,j, 1} }\right) + O(\e^{2})\\
&& =: {G}_{\e,1}^{ex}+ {G}_{\e, 2}^{ex} + {G}_{\e, 3}^{ex}+ {G}_{\e,4}^{ex} + O(\e^{2}).
\eeas

We can decompose $\Im{G}_{\e}^{ex}(x, x_0, k)$ into four parts accordingly and analyze each part one by one.

\begin{lem}
For $k\in W$, the following identities hold:
\beas
\Im{G}_{\e, 1}^{ex}(x, x_0, k)&=& \Im{G}^{ex}(x, x_0, k)= \f{\sin{k|x-x_0|}}{2\pi|x-x_0|}, \\
\Im{G}_{\e, 2}^{ex}(x, x_0, k)&= &-\e c_{\Lambda} \sum_{1\leq j \leq M}\f{\sin\left({k|x_0-z^{(j)}|+ k|x -z^{(j)}|}\right)}{2\pi|x_0-z^{(j)}|\cdot|x -z^{(j)}|}=  \e O(k),\\
\Im{G}_{\e, 3}^{ex}(x, x_0, k)
&= &-\f{(c_{\Lambda}\e)^{\f{3}{2}}}{\sqrt{|D|}} \left\{ \sum_{j=1}^M \zeta_j(0) \Im{\left[\f{1}{k-k_{0,\e,j,2}} \right]} -\sum_{j=1}^M \zeta_j(0) \Im{\left[ \f{1}{k-k_{0,\e,j,1}}\right]} \right\}\\
&&  +O(\e^{\f{3}{2}}) \left( O(k-k_{0,\e,j,1}) +O(k_{0,\e,j,1}) + O(k-k_{0,\e,j,2}) +O(k_{0,\e,j,2})\right),\\
\Im{G}_{\e, 4}^{ex}(x, x_0, k) &=& \sum_{1\leq j \leq M } \Im{ \left(\f{C_{j,2}}{k- k_{0,\e,j, 2} }+ \f{C_{j,1}}{k- k_{0,\e,j, 1} }\right)} + O(\e^{2}), \quad |k|\leq O(\sqrt{\e}),
\eeas
where $C_{j,2}$ and $C_{j,1}$ are independent of $k$ and are bounded by $O(\e^{2})$ and $\zeta_j(0)= \zeta_j(x, x_0, 0)$ (see (\ref{defzeta})).
\end{lem}

\begin{proof} The first two are obvious. We only need to consider the last two estimates. We first consider $\Im{G}_{\e, 3}^{ex}(x, x_0, k)$.
Note that
$$
\Im{G}_{\e, 3}^{ex}(x, x_0, k)
= -\f{(c_{\Lambda}\e)^{\f{3}{2}}}{\sqrt{|D|}}\sum_{j=1}^M \Im{\left[\f{1}{k-k_{0,\e,j,2}} \zeta_j(k)\right]}
 + \f{(c_{\Lambda}\e)^{\f{3}{2}}}{\sqrt{|D|}}\sum_{j=1}^M \Im{\left[ \f{1}{k-k_{0,\e,j,1}} \zeta_j(k)\right]},
$$
where $\zeta_j(k)=\zeta_j(x, x_0, k) $.

It is clear that $\zeta_j(k)= \zeta_j(k_{0,\e,j,2})+ O(k-k_{0,\e,j,2})$. Thus,
$$
\Im{\left[\f{1}{k-k_{0,\e,j,2}} \zeta_j(k)\right]} = \Im{\left[\f{\zeta_j(k_{0,\e,j,2})}{k-k_{0,\e,j,2}}\right]}+
O(k-k_{0,\e,j,2}).
$$

Since $\zeta_j(k_{0,\e,j,2})= \zeta_j(0) + O(k_{0,\e,j,2})$ and $\zeta_j(0)$ is a real number, we further get
$$
\Im{\left[\f{1}{k-k_{0,\e,j,2}} \zeta_j(k)\right]} = \zeta_j(0) \Im{\left[\f{1}{k-k_{0,\e,j,2}}\right]}+
O(k-k_{0,\e,j,2}) +O(k_{0,\e,j,2}).
$$

Similarly, we have
$$
\Im{\left[\f{1}{k-k_{0,\e,j,1}} \zeta_j(k)\right]} = \zeta_j(0) \Im{\left[\f{1}{k-k_{0,\e,j,1}}\right]}+
O(k-k_{0,\e,j,1}) +O(k_{0,\e,j,1}).
$$

This proves the estimate for $\Im{G}_{\e, 3}^{ex}(x, x_0, k)$.

Finally, we estimate the term
$$
\Im{G}_{\e, 4}^{ex}(x, x_0, k) = \sum_{1\leq j \leq M } \Im{ \left(\f{O(\e^{2})}{k- k_{0,\e,j, 2} }+ \f{O(\e^{2})}{k- k_{0,\e,j, 1} }\right)} + O(\e^{2}).
$$

Note that the terms of $O(\e^{2})$ in the above formula are in fact analytic functions for $k\in W$.
To analyze it, we rewrite one of the $O(\e^{2})$ terms as $h(k, \e)$. Then $h(k, \e)$ is of the order of $O(\e^{2})$. By using the formula
$$
\f{\p h}{\p k}(k, \e) = \f{1}{2\pi i}\int_{|z|=\f{1}{2}k_1}\f{h(z, \e)}{(z-k)^2}dz,
$$
we can derive that $\f{\p h}{\p k}(k, \e) =O(\e^{2})$ for $k \leq O(\sqrt{\e})$.
This further yields that
$$
h(k, \e) = h(k_{0, \e, j, 2}, \e) + O(\e^{2})(k - k_{0, \e, j, 2}), \quad k \leq O(\sqrt{\e}).
$$
It follows that
$$
\Im{\f{h(k, \e)}{k- k_{0,\e,j, 2}}} = \Im{\f{h(k_{0, \e, j, 2}, \e)}{k- k_{0,\e,j, 2}}} +  O(\e^{2}) =
\Im{\f{ C_{j,2}}{k- k_{0,\e,j, 2}}} +  O(\e^{2}),
$$
where $C_{j,2}$ is independent of $k$ and is bounded by  $O(\e^{2})$.
We can derive a similar result for the term $\Im{\f{O(\e^2)}{k- k_{0,\e,j, 2}}}$. This proves the estimate for $\Im{{G}_{\e,4}^{ex}(x, y_0, k)}$ and hence completes the proof of the Lemma.
\end{proof}

\medskip

Finally, by taking $k= \tau_1 \sqrt{\e} = \sqrt{\f{c_{\Lambda}\e}{|D|}}$ in the above lemma and using Proposition \ref{prop-resonances}, we obtain Theorem \ref{thm-super-resolution-1} on the imaging functional at a fixed frequency.

\subsection{Proof of Theorem \ref{thm-super-resolution-2}} \label{sec-sub-proof-super-2}
Recall that the imaging functional is given as follows:
$$
I = \int_0^{\infty} \Im{{G}_{\e}^{ex}(x, y_0, k)} \Im{\big(\check{f}(k)e^{ikt}\big)} d k,
$$
where $f(t) = \e^{\f{1}{4}} F(\e^{\f{1}{2}}t)$.

By a direct calculation, it is clear that
\be
\check{f}(k) = \e^{-\f{1}{4}} \check{F}(\e^{-\f{1}{2}}k).
\ee

We denote by $s(k)=s(k, t)= \Im{\big(\check{f}(k)e^{ikt}\big)} = \e^{-\f{1}{4}} \Im{\big( \check{F}(\e^{-\f{1}{2}}k)e^{ikt}\big)}$.
We first derive some basic properties about the input $s$ for the imaging functional.

\begin{lem}
The follow estimate holds:
$$
 \int_{0}^{\infty} |s(k)| dk \leq  O (\e^{\f{1}{4}}).
$$
\end{lem}
\begin{proof}
$$
 \int_0^{\f{1}{2}k_1} |s(k)| d k
\leq \int_{0}^{\f{1}{2}k_1}\e^{-\f{1}{4}} | \check{F}(\e^{-\f{1}{2}}k)| dk
\leq \e^{\f{1}{4}} \int_{0}^{\infty} |\check{F}(k)| dk.
$$
\end{proof}

\begin{lem} \label{lem-esti}
For $k , a \in [0,  \f{1}{2}k_1]$, the following estimate holds:
\be \label{esti}
|s(k) -s(a)| \leq O(\e^{-\f{3}{4}}) |k-a|.
\ee
\end{lem}
\begin{proof}
Since $\check{f}(k) = \e^{-\f{1}{4}} \check{F}(\e^{-\f{1}{2}}k)$,  we have $\check{f}'(k) = \e^{-\f{3}{4}} \check{F}'(\e^{-\f{1}{2}}k)$.
On the other hand,
$$
|\check{F}'(k)|  = |\int_{0}^{C_1} -t e^{-ikt} F(t) dt |  \leq \left(\int_{0}^{C_1} t^2 dt\right)^{\f{1}{2}} \cdot \|F\|_{L^2} \leq O(1).
$$

Thus, $\check{f}'(k)= O(\e^{-\f{3}{4}})$. It follows that $s'(t) = O(\e^{-\f{3}{4}})$ and hence the estimate (\ref{esti}) holds. The lemma is proved.\end{proof}

\medskip

To facilitate the analysis of the imaging functional $I$, we split it into five parts as follows:
\beas
I &=& \int_0^{\infty} \Im{{G}_{\e}^{ex}(x, y_0, k)} s(k) d k =
I_1 + I_2 + I_3 +I_4+ I_5,
\eeas
where
\beas
I_j &=& \int_0^{\f{1}{2}k_1} \Im{{G}_{\e, j}^{ex}(x, y_0, k)} s(k) d k, \quad 1\leq j \leq 4, \\
I_5 &=& \int_{\f{1}{2}k_1}^{\infty} \Im{{G}_{\e}^{ex}(x, y_0, k)} s(k) d k.
\eeas

We shall investigate each of the five terms above in the sequel.

We first consider
\beas
I_1= \int_0^{\infty}  \f{\sin{k|x-x_0|}}{2\pi|x-x_0|} s(k) dk.
\eeas
It is clear that $I_1$ is the imaging functional in the free space, and it yields the standard resolution. One cannot obtain super-resolution out of this term. We now estimate its magnitude.
The following is obvious:
$$
|I_1| \leq  \int_0^{\f{1}{2}k_1} |s(k)| dk =  O( \e^{\f{1}{4}}).
$$
However, if we impose additional smoothness conditions on the root signal $F$, we may obtain better estimates. Indeed, under Condition (\ref{cond-H2}), we have
$$
\int_{0}^{\infty} k | \check{F}(k) | dk  \leq \|\f{k}{k^2+1}\|_{L^2(0, \infty)} \cdot \|(k^2+1)\check{F}(k)\|_{L^2(0, \infty)}= O(1).
$$
It follows that
\beas
|I_1| \leq  \int_0^{\f{1}{2}k_1} k |s(k)| dk
\leq  \e^{\f{3}{4}} \int_0^{\infty} k |\check{F}( k )| d k
\leq   O( \e^{\f{3}{4}}).
\eeas

Under Condition (\ref{cond-quasi1}), it is clear that
$$
I_1= \int_0^{\e^{\f{1}{2}- \delta}} \f{\sin{k|x-x_0|}}{2\pi|x-x_0|} s(k) dk + o(\e^{\f{5}{4}}).
$$


The next objective of our study is to estimate
$$
I_2 = \e c_{\Lambda} \sum_{1\leq j \leq M}\int_0^{\infty} \left(\f{\sin\left({k|x_0-z^{(j)}|+ k|x -z^{(j)}|}\right)}{2\pi|x_0-z^{(j)}|\cdot|x -z^{(j)}|}\right)s(k)dk = M \cdot \int_0^{\infty} \e O(k) s(k) dk.
$$
By a similar argument as the one for $I_1$, we have the following estimate
\be
I_2 = M \cdot O(\e^{\f{7}{4}}).
\ee

%

We now consider
$$
I_3= \int_0^{\f{1}{2}k_1} \Im{G}_{\e, 3}^{ex}(x, x_0, k) s(k)dk,
$$
which can be further decomposed into the following four terms:
\beas
I_3 &=& -\int_0^{\f{1}{2}k_1} \f{(c_{\Lambda}\e)^{\f{3}{2}}}{\sqrt{|D|}} \sum_{j=1}^M \Im{\left[ \f{\zeta_j(0)}{k-k_{0,\e,j,2}}\right]}s(k)dk \\
&& + \int_0^{\f{1}{2}k_1} \f{(c_{\Lambda}\e)^{\f{3}{2}}}{\sqrt{|D|}} \sum_{j=1}^M \Im{\left[ \f{\zeta_j(0)}{k-k_{0,\e,j,1}}\right]}s(\Re{k_{0,\e,j,1}})dk  \\
&& + \int_{0}^{\f{1}{2}k_1}  \f{(c_{\Lambda}\e)^{\f{3}{2}}}{\sqrt{|D|}} \sum_{j=1}^M \Im{\left[ \f{\zeta_j(0)}{k-k_{0,\e,j,1}}\right]}\big( s(k) -s(\Re{k_{0,\e,j,1}}))  dk \\
&& + \int_0^{\f{1}{2}k_1} \big( O(k)+ O(\sqrt{\e})\big) O(\e^{\f{3}{2}}) s(k) dk \\
&=:& \sum_{j=1}^M I_{3, j, 2} + \sum_{j=1}^M I_{3, j, 1,1} + \sum_{j=1}^M I_{3, j, 1, 2} + O(\e^{\f{9}{4}}).
\eeas

It is clear that for $k > 0$, $\Im{ \left(\f{1}{k- k_{0,\e,j, 2} }\right)} =  O(\e^{-1}) \cdot O( |\Im{k_{0,\e,j, 2}}|)= O(\e)$.
Thus,
$$
|I_{3, j, 2}| \leq  O(\e^{\f{3}{2}})  \int_0^{\f{1}{2}k_1} O(\e) |s(k)| d k =
O(\e^{\f{5}{2}}) \cdot O(\e^{\f{1}{4}}) =
O(\e^{\f{11}{4}}).
$$

%



On the other hand, with the help of Lemma \ref{lem-integral},
we can deduce that
\beas
 I_{3, j, 1, 1}& =&  \f{(c_{\Lambda}\e)^{\f{3}{2}}\pi}{\sqrt{|D|}} \zeta_j(0) s(\Re{k_{0,\e,j,1}}) + O(\e^3)\\
 & =&  \f{(c_{\Lambda}\e)^{\f{3}{2}}\pi}{\sqrt{|D|}} \zeta_j(0) s(\tau_1 \sqrt{\e}) +O (\e^{\f{9}{4}}); \\
| I_{3, j, 1, 2}| & \leq &  \int_0^{\f{1}{2}k_1} \f{(c_{\Lambda}\e)^{\f{3}{2}}}{\sqrt{|D|}} | \Im{\left[ \f{\zeta_j(0)}{k-k_{0,\e,j,1}}\right]}|\cdot |k- \Re{k_{0,\e,j,1}}| O( \e^{-\f{3}{4}}) dk \\
& \leq &  O( \e^{\f{3}{4}}) \int_0^{\f{1}{2}k_1} | \Im{\left[ \f{1}{k-k_{0,\e,j,1}}\right]}|\cdot |k- \Re{k_{0,\e,j,1}}| dk \\
& \leq &  O( \e^{\f{3}{4}}) |\Im{k_{0,\e,j,1}}| \cdot |\ln{\e}| \\
& \leq &  O( \e^{\f{5}{2}}).
\eeas

Therefore, we get
\beas
I_3 &=& \f{(c_{\Lambda})^{\f{3}{2}}\pi}{\sqrt{|D|}}  \e^{\f{3}{2}} s(\tau_1 \sqrt{\e})\sum_{j=1}^M \zeta_j(0) + O(\e^{\f{9}{4}})\\
&=& \f{(c_{\Lambda})^{\f{3}{2}}\pi}{\sqrt{|D|}}  \e^{\f{3}{2}} s(\tau_1 \sqrt{\e})\sum_{j=1}^M\mathcal{G}(x, 0)^t Y_jY_j^t \mathcal{G}(x_0, 0) + O(\e^{\f{9}{4}})\\
&=& \f{(c_{\Lambda})^{\f{3}{2}}}{\sqrt{|D|}}  \e^{\f{3}{2}} s(\tau_1 \sqrt{\e})\mathcal{G}(x, 0)^t \mathcal{G}(x_0, 0) + O(\e^{\f{9}{4}})\\
&=& \f{(c_{\Lambda})^{\f{3}{2}}\pi}{\sqrt{|D|}} \e^{\f{3}{2}} s(\tau_1 \sqrt{\e})\sum_{j=1}^M G^{ex}(x, z^{(j)}, 0)G^{ex}(x_0, z^{(j)}, 0)+ O(\e^{\f{9}{4}})\\
&=& \f{(c_{\Lambda})^{\f{3}{2}}}{\sqrt{|D|}}  \e^{\f{5}{4}}  \Im{\big( \check{F}(\tau_1)e^{-i\tau_1 \sqrt{\e}t}\big)} \sum_{j=1}^M \f{1}{4\pi |x- z^{(j)}|\cdot |x_0- z^{(j)}| }+ O(\e^{\f{9}{4}}).
\eeas

We claim that the term $I_3$ is the main contribution to the super-resolution. Although it is of the order of $\e^{\f{5}{4}}$, its magnitude may be comparable to that of the term $I_1$ when $x$ and $x_0$ are close to one of the openings of the resonators.


We now consider
$$
I_4 = \int_0^{\f{1}{2}k_1} \Im{G}_{\e, 4}^{ex}(x, x_0, k) s(k)dk.
$$
We can decompose it into four parts:
\beas
I_4 &=&  \sum_{j=1}^M \int_0^{\f{1}{2}k_1} \Im{ \left(\f{O(\e^2)}{k- k_{0,\e,j, 2} }\right)} s(k)dk \\
 &&+   \sum_{j=1}^M \int_0^{2\tau_1 \sqrt{\e}} \Im{ \left(\f{C_{j,1}}{k- k_{0,\e,j, 1} } \right)} s(k)dk+
 \sum_{j=1}^M \int_{2\tau_1 \sqrt{\e}}^{\f{1}{2}k_1} \Im{ \left(\f{O(\e^2)}{k- k_{0,\e,j, 1} } \right)} s(k)dk \\
 && + \int_{0}^{\f{1}{2}k_1} O(\e^{2}) s(k)dk\\
 &=:&  \sum_{j=1}^M I_{4, j, 2}
 + \sum_{j=1}^M I_{4,j,1, 1} +  \sum_{j=1}^M I_{4,j, 1,2} +    O(\e^{2}).
\eeas
Since for $k >0$,  $\Im{ \left(\f{O(\e^2)}{k- k_{0,\e,j, 2} }\right)} = O(\e^{2}) O(\e^{-\f{1}{2}}) = O(\e^{\f{3}{2}})$,
thus,
$$
|I_{4, j, 2}| \leq O(\e^{\f{3}{2}})  \int_0^{\f{1}{2}k_1} |s(k)| d k =
O(\e^{\f{3}{2}}) \int_{0}^{\infty}\e^{-\f{1}{4}} | \check{F}(\e^{-\f{1}{2}}k)| dk
\leq O(\e^{\f{7}{4}}).
$$
Similarly, we can derive that
$$
|I_{4, j, 1, 2}| \leq O(\e^{\f{7}{4}}).
$$
On the other hand, let $a= \Re{k_{0,\e,j, 1}} = O(\sqrt{\e}), b= \Im{k_{0,\e,j, 1}} = O(\e^2)$, we have
\beas
I_{4, j,1, 1} &=&  \int_0^{2\tau_1 \sqrt{\e}} \Im{ \left(\f{C_{j,1}}{k- k_{0,\e,j, 1} } \right)} s(a)dk +
 \int_0^{2\tau_1 \sqrt{\e}} \Im{ \left(\f{C_{j,1}}{k- k_{0,\e,j, 1} } \right)} \big( s(k)- s(a)\big)dk \\
&& = I_{4, j, 1,1, 1} + I_{4, j, 1, 1, 2}.
\eeas
By Lemmas \ref{lem-integral} and \ref{lem-esti}, we can deduce that
$$
|I_{4, j, 1, 1, 1}| \lesssim |C_{j, 1} |\cdot |s(a)| = O(\e^{\f{7}{4}}), \quad
|I_{4, j, 1,1,  2}| \leq  |C_{j, 1} | O(\e^{-\f{1}{4}})  = O(\e^{\f{7}{4}}).
$$

Therefore,
$$
 I_4   = O(\e^{\f{7}{4}}),
$$
and we can conclude that $I_4$ is dominated by $I_3$, though it may also yield a resolution of $O(1)$-scale.

Finally, we consider
$$
I_5 = \int_{\f{1}{2}k_1}^{\infty} \Im{G}_{\e}^{ex}(x, x_0, k) s(k)dk
= \e^{\f{1}{4}}\int_{\f{k_1}{2\sqrt{\e}}}^{\infty}  \Im{G}_{\e}^{ex}(x, x_0, \sqrt{\e}k)\Im{\big( \check{F}(\tau_1)e^{-i\tau_1 \sqrt{\e}t}\big)} dk.
$$

It is clear that $I_5$ can yield resolution on the scale of order $O(1)$. In order to claim super-resolution,
the condition (\ref{cond-quasi2}) must be assumed (by comparing the magnitude of the term $I_3$). This completes the proof of Theorem \ref{thm-super-resolution-2}.

\medskip

Finally, we have two remarks about the condition (\ref{cond-quasi1}) and (\ref{cond-quasi2}).

\begin{rmk}
The super-resolution property requires that sub-wavelength details can be resolved. In our case, the signal $f$ will be mainly concentrated in a $O(\sqrt{\e})$ neighborhood of the origin in the frequency domain, while the resolution achieved is on a $O(1)$ scale.  Thus we need to impose decay conditions on the function $s(k)$ for $k\in [O(\sqrt{\e}), \f{1}{2}k_1]$, which will be equivalent to the smoothness conditions on the root-signal $F$.
For this purpose, let $\delta \in (0, \f{1}{2})$ be a fixed number. Note that
$$
I_1= \int_0^{\e^{\f{1}{2}- \delta}} \f{\sin{k|x-x_0|}}{2\pi|x-x_0|} s(k) dk +
\int_{\e^{\f{1}{2}- \delta}}^{\f{1}{2}k_1}  \f{\sin{k|x-x_0|}}{2\pi|x-x_0|} s(k) dk=:
I_{1,1}+ I_{1, 2}.
$$

The term $I_{1, 2}$ yields resolution on the scale of $O(1)$. Thus we need to suppress it in order to claim super-resolution. Observe that
\beas
I_{1, 2} \lesssim  \int_{\e^{\f{1}{2}- \delta}}^{\f{1}{2}k_1} |s(k)| dk  \lesssim  \e^{\f{1}{4}} \int_{\e^{- \delta}}^{\infty} |\check{F}(k)| dk
 \lesssim  \e^{\f{1}{4} + n \delta}  \int_{\e^{- \delta}}^{\infty} k^n |\check{F}(k)| dk.
\eeas
By imposing smoothness conditions on the root signal $F$, we can make the term $I_{1, 2}$ arbitrarily small and hence suppress the resolution due to frequencies outside of the quasi-stationary regime, say $k\in [O(\sqrt{\e}), \infty]$. Especially, this is the case when we assume that
$$
F \in \mathcal{C}^{\infty}_0([0, C_1]).
$$
\end{rmk}

\begin{rmk}
A complete analysis of $I_5$ may involve a detailed analysis of the Green function ${G}_{\e}^{ex}(x, x_0, k)$ for $k \geq \f{1}{2}k_1$, which we believe can be achieved by the method developed in Sections \ref{sec-resonator-single} and \ref{sec-resonator-multi}. An heuristic procedure can be carried out as follows. We first partition the Fourier domain $[\f{1}{2}k_1, \infty]$ into pieces of intervals based on the distribution of the eigenvalues for the Neumann problem in the domain $D$, and analyze ${G}_{\e}^{ex}(x, x_0, k)$ for $k$ in each interval can derive estimate as in Theorem \ref{thm-multi}. We then analyze the integration of the integrand of $I_5$ in each interval by
using the same approach as we did for the imaging functionals $I_1$, $I_2$, $I_3$, and $I_4$. Finally, we sum up the contributions to $I_5$ from each intervals to obtain a global estimate.
\end{rmk}

\section{Concluding remarks}

In this paper, we have for  the first time established a  mathematical theory of super-resolution in the context of a deterministic medium. We have highlighted  the  mechanism of super-resolution and super-focusing for waves  using sub-wavelength-scaled resonant deterministic media. When a system of sub-wavelength resonators is excited by  a point source
at resonant frequencies, the information on the point source is encoded into the various resonant modes of the system of resonators. Resonant modes can propagate into the far-field, which reveals the information on the source. As a result, we can obtain super-resolution which is only limited by the distance between resonators and the signal-to-noise ratio in the data.
The system of resonators acts in some sense as an array of sub-wavelength sensors.

Our approach opens many new avenues for mathematical imaging and focusing in resonant media.
Many challenging problems are still to be solved. From noisy data, it is very challenging to precisely quantify the super-resolution enhancement factor in terms of the signal-to-noise in the data.  It would be also very interesting to use resonant media for shaping and compressing waves. Moreover, building resonant media using negative parameter materials is an actual problem of great interest \cite{fink13}. Finally, it is expected that our approach can be generalized to justify the fact that super-resolution can be achieved using structured light illuminations \cite{structured, structured2}.

\appendix
\section{Some integration formulas and estimate} \label{appendixA}

\begin{lem} \label{lem-integral}
Let $A_1, A_2, a, b$ be real numbers. Assume that $A_1 \leq a \leq A_2$ and $b\neq 0$, then
\bea
\int_{A_1} ^{A_2} \f{1}{ k- a -bi} dk &=& \f{1}{2}\ln{\f{(A_2 -a)^2+ b^2}{(A_1 -a)^2+ b^2}} + i \int_{\f{A_1-a}{b}}^{\f{A_2-a}{b}} \f{1}{t^2 +1} dt, \\
\int_{A_1} ^{A_2} |\f{1}{ k- a -bi}|\cdot |k-a| dk &=& 2(A_2 -A_1 - 2b), \\
\int_{A_1} ^{A_2} | \Im{ ( \f{1}{ k- a -bi})}|\cdot |k-a| dk &=& |b|\big(\ln{|A_2 -a|} +\ln{|A_1 -a|} - 2 \ln{|b|}\big),\\
\int_{A_1} ^{A_2} | \Im{ ( \f{1}{ k- a -bi})}| dk &=& \ln{|A_2 -a|} +\ln{|A_1 -a|} - 2 \ln{|b|}.
\eea
\end{lem}

\section{A criterion of recording time for the time-reversal experiment} \label{appendixB}
We derive a heuristic criterion for the recording time for the time-reversal experiment in section \ref{sec-sub-TRE}. A rigorous estimate requires full knowledge of the local energy-decaying rate for wave fields which in turn is connected to the distribution of the imaginary part of all resonances for the system of resonators, and is out of the scope of the paper. We refer to \cite{BF02, B98} for some discussions.

We first estimate of a local energy-decaying rate for the wave field $u(x,t)$.
Recall that
$$
\check{u}(x, k) = \check{G}_{\e}(x, y_0, k)\check{f}(k).
$$
We assume that the signal $f(\cdot)$ is supported in $(0, C_1 \e^{-\f{1}{2}})$ for some constant $C_1$ and $\check{f}$ is essentially supported in the quasi-stationary regime. By theorem \ref{thm-multi}, we may use the following approximation
$$
\check{u}(x, k)\approx \check{u}_1(x, k)+\check{u}_2(x, k)+ \check{u}_3(x, k)+ \check{u}_4(x, k),
$$
where
\beas
& \ds \check{u}_1(x, k) =
\check{G}^{ex}(x, x_0, k) \check{f}(k),\\
\nm
& \ds \check{u}_2(x, k) =  {\e}c_{\Lambda}\sum_{1\leq j \leq M}\check{G}^{ex}(z^{(j)}, x_0, k)\check{G}^{ex}(x, z^{(j)}, k)\check{f}(k) \\
\nm
& \ds \check{u}_3(x, k) =  \sum_{j=1}^M \left( \f{1}{k-k_{0,\e,j,2}} -\f{1}{k-k_{0,\e,j,1}}\right)\f{(c_{\Lambda}\e)^{\f{3}{2}}}{\sqrt{|D|}} \zeta_j(k) \check{f}(k) \\
\nm & \ds \check{u}_4(x, k) =  \sum_{1\leq j \leq M } \left(\f{O(\e^{2}) \check{f}(k)}{k- k_{0,\e,j, 2} }+ \f{O(\e^{2}) \check{f}(k)}{k- k_{0,\e,j, 1} }\right).
\eeas

We denote by
$$
u_j(x, t) =  \int e^{-ikt} \check{u}_j(x, k) dk, \quad 1\leq j \leq 4,
$$
and analyze each of the four terms.
First, we have
\beas
u_1(x, t) = \int e^{-ikt} \f{e^{ik |x-x_0|}}{2\pi |x-x_0|} \check{f}(k)dk = \f{f(|x-x_0|-t)}{|x-x_0|}.
\eeas

Second,
\beas
u_2(x, t)
& =& \sum_{1\leq j \leq M } \int e^{-ikt} \f{e^{ik |x-z^{(j)}|+ ik |x_0-z^{(j)}|}}{4\pi^2 |x-z^{(j)}|\cdot |x_0-z^{(j)}| }\check{f}(k)dk \\
& = &\sum_{1\leq j \leq M } \f{f(|x-z^{(j)}|+|x_0-z^{(j)}|-t)}{4\pi^2 |x-z^{(j)}|\cdot |x_0-z^{(j)}|}.
\eeas

It is clear that for $t$ sufficiently large, say $$t \geq C_1 \e^{-\f{1}{2}} + \max_{1\leq j \leq M } \bigg\{ |x-z^{(j)}|+|x_0-z^{(j)}| \bigg\},$$ we have
$$
u_1(x, t) = u_2(x, t)= 0.
$$

We now estimate $u_3(x, t)$.
Note that both $\check{G}(x, y_0, \cdot)$ and $\check{f}(\cdot)$ are analytic in the lower half of the complex plane. By applying the Residue theorem, we can obtain
\beas
u_{3}(x, t) &=& \int e^{-ikt} \check{u}_3(x, k) dk \\
&& =  \f{(c_{\Lambda}\e)^{\f{3}{2}}}{\sqrt{|D|}} \sum_{j=1}^M \left( e^{-ik_{0, \e, j, 2}t} \check{f}(k_{0, \e, j, 2})\zeta_j(k_{0, \e, j, 2}) -e^{-ik_{0, \e, j, 1}t} \check{f}(k_{0, \e, j, 1})\zeta_j(k_{0, \e, j, 1})\right)
\eeas

Since $\Im{k_{0, \e, j, 2} } = \Im{k_{0, \e, j, 2} } = O(\e^2)$, we can derive that
$$
|u_3(x, t)| = O(e^{-\e^2 t}) ,\quad t>0.
$$

Similarly, we can show that
$$
|u_4(x, t)| = O(e^{-\e^2 t}) ,\quad t>0.
$$

We assume that the source location $x_0$ and imaging point $x$ are not too far away from the apertures of the resonators. Then the fields $u_1$ and $u_2$ vanishes for large time $T$.  Thus we can conclude that
\be  \label{esti-decay1}
|u(x, t)| = O(e^{-\e^2 t}) ,\quad t>0.
\ee

In a similar way, we can derive that
\be  \label{esti-decay2}
|u_{t}(x, t)| = O(e^{-\e^2 t}) ,\quad t>0.
\ee

We remark that the estimate (\ref{esti-decay1}) and (\ref{esti-decay1}) for the the wave field $u(x, t)$ have excluded contributions from those outside the quasi-stationary regime, which we assume to be negligible.

We next investigate the term
\beas
\Theta(x,t)= \int_{\Omega}d y\left(u_t(y, T) G_{\e}(y, x, t)  -\f{\p G_{\e}}{\p t}(y, x, t)u(y, T)\right).
\eeas

It is clear that $\Theta$ satisfies the following wave equation
\beas
\f{\p^2 \Theta}{\p t^2}(x,t) - \Delta \Theta (x,t) = u_t(y, T) \delta(t) - u(y, T) \delta'(t), \quad (x, t)\in \Om_{\e} \times \R,
\eeas
which is equivalent to
\beas
\f{\p^2 \Theta}{\p t^2}(x,t) - \Delta \Theta (x,t) &=& 0, \quad (x, t)\in \Om_{\e} \times (0, \infty),\\
\Theta (x,0)&=& -u(x, T), \quad x \in \Om_{\e}, \\
\f{\p \Theta}{\p t}(x,0) &= & u_t(x, T), \quad x \in \Om_{\e}. \\
\eeas

By the local energy-decaying property of wave fields and the estimate (\ref{esti-decay1}) and (\ref{esti-decay1}), we may that argue
$$
| \Theta(x, t)| \lesssim  O(e^{-\e^2 T}), \quad t >0.
$$

Finally, by the analysis above, we obtain the following criterion for the recording time $T$ for the time-reversal experiment
\be
T \gg  \e^{-2}.
\ee

\end{document}